\documentclass[12pt]{amsart}
\usepackage{amsmath}
\usepackage{amsfonts}
\usepackage{amssymb}
\usepackage{amscd}
\usepackage{mathtools}
\usepackage{amsxtra}
\usepackage{amsthm}
\usepackage{graphicx}
\usepackage[abbrev,alphabetic]{amsrefs}
\RequirePackage[dvipsnames,usenames]{color}
\usepackage{soul,xcolor}
\setstcolor{red}
\usepackage{stmaryrd}
\usepackage{booktabs}
\usepackage{multirow}
\newtagform{tiny}{\tiny(}{)}
\usepackage{aliascnt}

\usepackage{mathtools}
\usepackage{hyperref}
\usepackage[margin=1.25in]{geometry}
\usepackage{mathrsfs}

\usepackage{amsthm}
\usepackage{comment}
\usepackage[all,cmtip]{xy}
\usepackage{tikz-cd}
\usetikzlibrary{cd}

\tikzcdset{
  cells={font=\everymath\expandafter{\the\everymath\displaystyle}},
}

\usepackage[all]{xy}

\usepackage{cleveref}

\makeatletter
\def\@tocline#1#2#3#4#5#6#7{\relax
  \ifnum #1>\c@tocdepth 
  \else
    \par \addpenalty\@secpenalty\addvspace{#2}%
    \begingroup \hyphenpenalty\@M
    \@ifempty{#4}{%
      \@tempdima\csname r@tocindent\number#1\endcsname\relax
    }{%
      \@tempdima#4\relax
    }%
    \parindent\z@ \leftskip#3\relax \advance\leftskip\@tempdima\relax
    \rightskip\@pnumwidth plus4em \parfillskip-\@pnumwidth
    #5\leavevmode\hskip-\@tempdima
      \ifcase #1
       \or\or \hskip 1em \or \hskip 2em \else \hskip 3em \fi%
      #6\nobreak\relax
    \hfill\hbox to\@pnumwidth{\@tocpagenum{#7}}\par
    \nobreak
    \endgroup
  \fi}
\makeatother

\newcommand{\Z}{\mathbb{Z}}

\newcommand{\perf}{\mathrm{perf}}

\makeatletter
\DeclareRobustCommand{\graded}{%
    \@ifnextchar\bgroup{\graded@with}{\ensuremath{\mathrm{gr}}}%
}
\newcommand{\graded@with}[1]{\ensuremath{#1\mathrm{-gr}}}
\makeatother
\newcommand{\grmapspt}{\operatorname{map}^{\graded}}
\DeclareMathOperator*{\grlim}{grlim}
\DeclareMathOperator*{\grcolim}{grcolim}
\newcommand{\grotimes}{\mathbin{\otimes}^{\graded}}
\newcommand{\Lgrotimes}{\mathbin{\otimes}^{L\graded}}

\newcommand{\cLgrotimes}{\mathbin{\widehat{\otimes}}^{L\graded}}

\makeatletter
\DeclareRobustCommand{\comp}[1]{%
    \@ifnextchar\bgroup{\comp@with{#1}}{\comp@without{#1}}%
}
\newcommand{\comp@with}[2]{\Lambda_{#1}(#2)}
\newcommand{\comp@without}[1]{\ensuremath{#1\mathrm{-comp}}}
\makeatother
\newcommand{\dcomp}[2]{d\Lambda_{#1}(#2)}
\newcommand{\grcomp}[2]{\Lambda^{\graded}_{#1}(#2)}
\newcommand{\dgrcomp}[2]{d\Lambda^{\graded}_{#1}(#2)}
\DeclareMathOperator{\grwedge}{{\wedge \graded}}

\DeclareMathOperator{\coMod}{coMod}

\DeclareMathOperator{\Supp}{Supp}

\DeclareMathOperator{\Coker}{Coker}

\DeclareMathOperator{\Ker}{Ker}

\DeclareMathOperator{\Ind}{Ind}
\DeclareMathOperator{\Pro}{Pro}

\theoremstyle{plain}
\newtheorem{theorem}{Theorem}[section]

\newtheorem{theoremA}{Theorem}

\crefname{theoremA}{Theorem}{Theorems}
\Crefname{theoremA}{Theorem}{Theorems}

\newaliascnt{lemma}{theorem}
\newtheorem{lemma}[lemma]{Lemma}
\aliascntresetthe{lemma}
\crefname{lemma}{Lemma}{Lemmas}
\Crefname{lemma}{Lemma}{Lemmas}

\newaliascnt{proposition}{theorem}
\newtheorem{proposition}[proposition]{Proposition}
\aliascntresetthe{proposition}
\crefname{proposition}{Proposition}{Propositions}
\Crefname{proposition}{Proposition}{Propositions}

\newaliascnt{corollary}{theorem}
\newtheorem{corollary}[corollary]{Corollary}
\aliascntresetthe{corollary}
\crefname{corollary}{Corollary}{Corollaries}
\Crefname{corollary}{Corollary}{Corollaries}

\newaliascnt{claim}{theorem}
\newtheorem{claim}[claim]{Claim}
\aliascntresetthe{claim}
\crefname{claim}{Claim}{Claims}
\Crefname{claim}{Claim}{Claims}

 \newtheorem*{claim*}{Claim}

\theoremstyle{definition}
\newaliascnt{definition}{theorem}
\newtheorem{definition}[definition]{Definition}
\aliascntresetthe{definition}
\crefname{definition}{Definition}{Definitions}
\Crefname{definition}{Definition}{Definitions}

\newaliascnt{notation}{theorem}
\newtheorem{notation}[notation]{Notation}
\aliascntresetthe{notation}
\crefname{notation}{Notation}{Notations}
\Crefname{notation}{Notation}{Notations}

\newaliascnt{example}{theorem}

\aliascntresetthe{example}
\crefname{example}{Example}{Examples}
\Crefname{example}{Example}{Examples}

\theoremstyle{remark}
\newaliascnt{remark}{theorem}
\newtheorem{remark}[remark]{Remark}
\aliascntresetthe{remark}
\crefname{remark}{Remark}{Remarks}
\Crefname{remark}{Remark}{Remarks}

\newaliascnt{setting}{theorem}

\aliascntresetthe{setting}
\crefname{setting}{Setting}{Settings}
\Crefname{setting}{Setting}{Settings}

\newaliascnt{construction}{theorem}
\newtheorem{construction}[construction]{Construction}
\aliascntresetthe{construction}
\crefname{construction}{Construction}{Constructions}
\Crefname{construction}{Construction}{Constructions}

\newenvironment{claimproof}[0]
  {%
   \paragraph{\textit Proof.}%
  }
  {%
    \hfill$\blacksquare$%
  }

\numberwithin{equation}{section}



\usepackage{mymacros_common_Ryo, mymacros_english_paper_Ryo}

\title{Derived graded modules}
\author{Ryo Ishizuka}
\address{Institute of Science Tokyo, Tokyo 152-8551, Japan}
\email{ishizuka.r.ac@m.titech.ac.jp}

\author{Shou Yoshikawa}
\address{Institute of Science Tokyo, Tokyo 152-8551, Japan}
\email{yoshikawa.s.9fe9@m.isct.ac.jp}



\begin{document}

\begin{abstract}
We introduce the notion of the \(\infty\)-category of (complete) derived $G$-graded modules over a $G$-graded ring \(R\) for a torsion-free abelian group $G$, and we study its foundational properties.
Moreover, we prove a categorical equivalence between (complete) derived \(G\)-graded modules over \(R\) and derived (formal) comodules over a certain comonad constructed from the group ring \(R[G]\) of \(G\) over \(R\).
\end{abstract}

\maketitle

\tableofcontents

\section{Introduction}

Graded modules over a graded ring appear in various mathematical contexts, such as algebraic geometry, commutative algebra, algebraic topology, and representation theory. For instance, one of the most fundamental correspondences is that between graded modules over a graded ring and quasi-coherent sheaves on the projective spectrum of the ring.

To provide a flexible framework for modern applications, it is natural to generalize these notions to an $\infty$-categorical setting.
In this direction, previous work has developed the theory of derived $\setZ$-graded modules over rings with trivial grading, especially in connection with filtrations and Rees constructions (see, for example, \cites{lurie2015Rotation,antieau2025Filtrations,moulinos2021Geometry,raksit2026Hochschild}\footnote{More generally, these works consider the $\infty$-category of $\setZ$-graded objects in a given $\infty$-category.}).
However, many geometric contexts---such as projective varieties in mixed characteristic---require a theory that accommodates both nontrivial gradings on the base ring and more general grading groups, such as $\setZ^2$ and $\setZ[1/p]$.





For the purposes of this introduction, let $R$ be a $G$-graded ring, where $G$ is a torsion-free abelian group, and let $I$ be a finitely generated homogeneous ideal of $R$.
In this paper, we introduce the $\infty$-category $\mcalD_{\graded{G}}(R)$ of \emph{derived $G$-graded modules} and its full subcategory $\mcalD_{\graded{G}}^{\comp{I}}(R)$ of \emph{derived gradedwise $I$-complete}\footnote{In our previous work \cite{ishizuka2025Graded}, we developed a graded variant of perfectoid rings, called \emph{graded perfectoid rings}.
For a $G$-graded ring $R$ to be graded perfectoid, one needs $1/p \in G$ except in trivial cases, so we must consider general $G$ rather than only $\setZ$. Moreover, in order to define graded perfectoid rings, we need the notion of gradedwise completeness, which is a counterpart of $p$-adic completeness for perfectoid rings.} objects, and we establish foundational properties of these categories that are familiar from the non-graded or non-derived setting:

\begin{enumerate}
    \item The category $\mcalD_{\graded{G}}(R)$ is a compactly generated, stable, closed symmetric monoidal $\infty$-category (\Cref{PropertiesDerivedGradedModules}).
    \item It is equivalent to the derived $\infty$-category $\mcalD(\Mod_{\graded{G}}(R))$ of $G$-graded $R$-modules (\Cref{DerivedGradedCatEquiv}).
    \item The inclusion functor \(\mcalD_{\graded{G}}^{\comp{I}}(R) \hookrightarrow \mcalD_{\graded{G}}(R)\) admits a left adjoint, which is called the derived gradedwise completion functor \(\dgrcomp{I}{-}\) (\Cref{DerivedGradedwiseComp}).
    \item The functor \(\dgrcomp{I}{-}\) can be written as the limit of derived quotients (\Cref{ind-generator}).
    \item A Nakayama-type result holds for derived gradedwise completeness (\Cref{graded-derived-nakayama}).
\end{enumerate}

Our goal is not only to establish these foundations but also to clarify the relationship between derived graded modules and derived comodules.
Classically, a $G$-grading on an abelian group corresponds to a $\setZ[G]$-coaction; for example, $\setZ$-graded abelian groups correspond to abelian groups equipped with a $\setG_m$-action.
We obtain a generalization of this correspondence in our setting (\Cref{intro}):



\begin{theoremA}[{\Cref{ComonadicityGradedModules}}]\label{intro}
    Let \(\rho_R \colon R \to R[G]\) be the \(\setZ[G]\)-coaction induced by the \(G\)-grading on \(R\) (see \Cref{def-coaction}).
    Then there exists an equivalence of \(\infty\)-categories
    \begin{align*}
        \mcalD_{\graded{G}}^{\comp{I}}(R) & \xrightarrow{\simeq} \coMod_G^{\comp{I}}(R); \\ M & \mapsto (\dcomp{I}{M}, \rho \colon \dcomp{I}{M} \to \dcomp{I}{M[G]}),
    \end{align*}
    where \(\dcomp{I}{-} \colon \mcalD(R) \to \mcalD^{\comp{I}}(R)\) is the derived \(I\)-completion functor, the derived \(R\)-module \(M[G]\) is defined by \(\rho_{R,*}(M \otimes_R^L R[G])\), and \(\coMod_G^{\comp{I}}(R)\) is the \(\infty\)-category of comodules over the comonad \(\mcalT^I \colon M \mapsto \dcomp{I}{M[G]}\) on \(\mcalD^{\comp{I}}(R)\).
    \footnote{As we explain in \Cref{RemarkComoduleNotation}, the \(\infty\)-category \(\coMod_G^{\comp{I}}(R)\) is not necessarily the \(\infty\)-category of comodules over the coalgebra \(\dcomp{I}{R[G]}\) in \(\mcalD^{\comp{I}}(R)\) unless the grading on \(R\) is trivial.}
\end{theoremA}

\noindent
The main input in the proof of \Cref{intro} is the coherent splitting of the unit of the completed forgetful adjunction (\Cref{CofreeCompletionCompatibility}). Together with \Cref{GeneralSplitUnitTheorem,SplitUnitComonadicity}, it gives both \Cref{intro} and the limit comparison criterion of \Cref{LimitCommutative} for the functor\footnote{Note that the functor $\mcalF^I$ does not preserve limits in general (see \Cref{NonIsomGradedLimit}).}
\[
\mcalF^I \colon \mcalD_{\graded{G}}^{\comp{I}}(R) \xrightarrow{\mathrm{forget}} \mcalD(R) \xrightarrow{\dcomp{I}{-}} \mcalD^{\comp{I}}(R).
\]


\subsection*{Structure of this paper}

This paper is organized as follows. In \Cref{SectionPreliminaries}, we fix notation and recall basic properties of derived quotients and comodules over coalgebras.
In \Cref{SectionDerivedGradedModules}, we define the \(\infty\)-category of derived graded modules over graded rings and study its basic properties.
In \Cref{SectionDerivedGradedwiseComp}, the notion of derived gradedwise completion is introduced.
\Cref{SectionRetraction} establishes the general consequences from the existence of a retraction of a unit of adjunctions.
Finally, \Cref{SectionBBL} proves our main theorem (\Cref{intro,ComonadicityGradedModules}).

\subsection*{Acknowledgments}
The authors thank the anonymous referee for identifying errors concerning \(\infty\)-categorical limits in an earlier version of this paper.
We use ChatGPT to correct these errors in \Cref{SectionRetraction} and \Cref{SectionBBL}.

The first-named author was supported by JSPS KAKENHI Grant number 24KJ1085.
The second-named author was supported by JSPS KAKENHI Grant number JP24K16889.

\section{Preliminaries} \label{SectionPreliminaries}

In this section, we fix notation and recall some definitions and lemmas on derived quotients and coactions of coalgebras that will be used in later sections.

\subsection{Notation and terminology}

Throughout this paper, we use the following notation and terminology.

\subsubsection{Graded rings} \label{GradedNotation}
\begin{enumalphp}
    \item Let \(G\) be a torsion-free abelian group with identity element \(0\).
    \item A \emph{\(G\)-graded ring} is a pair of a commutative ring \(R\) and additive subgroups \(\{R_g\}_{g \in G}\) such that \(R\) is decomposed as the coproduct \(R = \bigoplus_{g \in G} R_g\) and satisfies \(R_g R_{g'} \subseteq R_{g + g'}\) for any \(g, g' \in G\). A \emph{homogeneous element} of \(R\) is an element in \(R_g\) for some \(g \in G\) and a \emph{homogeneous ideal} is an ideal of \(R\) which is generated by homogeneous elements.
    The subset $\{g \in G \mid R_g \neq 0\}$ of $G$ is called the \emph{support of $R$} and is denoted by $\Supp{R}$.
    Furthermore, for a submonoid $H$ of $G$, if the support of $R$ is contained in $H$, then we say that $R$ is an \emph{$H$-graded ring}.
    \item A \emph{\(G\)-graded ring morphism} between \(G\)-graded rings \(R = \bigoplus_{g \in G} R_g\) and \(S = \bigoplus_{g \in G} S_g\) is a ring homomorphism \(\varphi \colon R \to S\) such that \(\varphi(R_g) \subseteq S_g\) for any \(g \in G\).
\end{enumalphp}

\subsubsection{Higher categorical stuff}
\begin{enumalphp}
    \item We work with \(\infty\)-categories (i.e., \((\infty,1)\)-categories), using the model of quasicategories. Our main references are \cites{lurie2009Higher, lurie2017Higher, lurie2018Spectral, Kerodon}. We freely use standard definitions and basic results from these references without further comment.
    \item A ring or module is called \emph{discrete} if it has no higher homotopy groups.
    \item For (discrete) \(G\)-graded ring \(R\), we will define \emph{the \(\infty\)-category \(\mcalD_{\graded{G}}(R)\) of graded \(R\)-modules} in \Cref{SectionDerivedGradedModules}.
    \item In \(\mcalD_{\graded{G}}(R)\) of a \(G\)-graded ring \(R\), the (co)limits and symmetric products are written by \(\grlim\), \(\grcolim\), and \(- \Lgrotimes -\) (if they exist) but in \(\mcalD(R)\) of a discrete ring \(R\), i.e., derived limits and derived tensor products, they are often written by \(\lim\) (or \(R\lim\)), \(\colim\) and \(- \otimes^L -\).
\end{enumalphp}

\subsubsection{Completion functors}

Let $R$ be a ring and let $I$ be a finitely generated ideal of $R$.
Let $S$ be a $G$-graded ring and let $J$ be a finitely generated homogeneous ideal.
\begin{enumalphp}
    \item The symbol \(\comp{I}{-}\) is only used to denote the (classical) \(I\)-adic completion of discrete rings and modules.
    \item The symbol \(\dcomp{I}{-}\) is used to denote the derived (\(I\)-)completion of objects in a derived category such as \(\mcalD(R)\).
    \item The symbol \(\grcomp{J}{-}\) is the \emph{gradedwise (\(J\)-)completion} of discrete graded rings and modules which is defined in \cite{ishizuka2025Graded}*{Construction 3.3}. In loc. cit., it is denoted by \((-)^{\grwedge}\) but in this paper, we use the symbol \(\grcomp{J}{-}\) to avoid confusion with other derived completion functors.
    \item The symbol \(\dgrcomp{J}{-}\) is the \emph{derived gradedwise \(J\)-completion} of objects in \(\mcalD_{\graded{G}}(S)\), which is defined in \Cref{DerivedGradedwiseComp}.
    \item The full subcategory of \emph{derived \(I\)-complete} objects in \(\mcalD(R)\) is denoted by \(\mcalD^{\comp{I}}(R)\) and the full subcategory of \emph{derived gradedwise \(J\)-complete} objects in \(\mcalD_{\graded{G}}(S)\) is denoted by \(\mcalD_{\graded{G}}^{\comp{J}}(S)\) (\Cref{DerivedGradedwiseComp}).
    \item We will denote the derived (gradedwise) completed tensor products by \(- \widehat{\otimes}^L -\) and \(- \cLgrotimes -\) in \(\mcalD^{\comp{I}}(R)\) and \(\mcalD_{\graded{G}}^{\comp{J}}(S)\) respectively.
\end{enumalphp}

\subsection{Derived quotients}

We fix some notation and recall several results on derived quotients that will be used in this paper.

\begin{definition}
Let $R$ be a ring and let $I$ be a finitely generated ideal of $R$.
\begin{itemize}
    \item For an $R$-module $M$, we say that $M$ is  \emph{$I^\infty$-torsion} if for every $m \in M$, there exists $n \in \Z_{>0}$ such that $I^n m=0$.
    \item For $M \in \mcalD(R)$, we say that $M$ is \emph{cohomologically $I^\infty$-torsion} if $\pi_i(M)$ is $I^\infty$-torsion for every $i \in \Z$.
\end{itemize}
\end{definition}

\begin{lemma}\label{extension-I-infty-torsion}
Let $R$ be a ring and let $I$ be a finitely generated ideal of $R$.
Consider a short exact sequence of $R$-modules
\[
0 \to M_1 \to M_2 \to M_3 \to 0.
\]
If $M_1$ and $M_3$ are $I^\infty$-torsion, then so is $M_2$.
\end{lemma}

\begin{proof}
Take generators $f_1,\ldots,f_r$ of $I$ and $a \in M_2$.
Since $M_3$ is $I^\infty$-torsion, there exists $n>0$ such that
\[
f_i^n a \in \Ker(M_2 \to M_3) \cong M_1
\]
for all $1 \leq i \leq r$.
As $M_1$ is $I^\infty$-torsion, there exists $n'>0$ such that $f_i^{n+n'}a=0$ for all $1 \leq i \leq r$.
Hence $I^{r(n+n')}a=0$, so $M_2$ is $I^\infty$-torsion, as desired.
\end{proof}

\begin{definition}
Let $R$ be a ring, $M \in \mathcal{D}(R)$, and $f_1,\ldots,f_r \in R$.
We define the \emph{derived quotient of $M$ by $f_1,\ldots,f_r$} as
\[
M/^L(f_1,\ldots,f_r)  \defeq  M \otimes^L_{R[X_1,\ldots,X_r]} R[X_1,\ldots,X_r]/(X_1,\ldots,X_r),
\]
where $M$ is regarded as an object of $\mcalD(R[X_1,\ldots,X_r])$ via the $R$-algebra homomorphism
\[
R[X_1,\ldots,X_r] \to R, \quad X_i \mapsto f_i.
\]
In particular,
\[
M/^L(f_1,\ldots,f_r) \simeq M \otimes^L_{R} R/^L(f_1,\ldots,f_r)
\]
holds in \(\mcalD(R)\).
\end{definition}

\begin{proposition}\label{to-nakayama}
Let $R$ be a ring, let $f_1,\ldots,f_r \in R$, and set $I=(f_1,\ldots,f_r)$.
Let $M \in \mcalD(R)$ be a connective and bounded object.
Then $M/^L(f_1,\ldots,f_r)$ is connective, bounded, and cohomologically $I^\infty$-torsion.
\end{proposition}

\begin{proof}
We proceed by induction on $r$.

For $r=1$, set $f \defeq f_1$.
From the fiber sequence
\[
R[X] \xrightarrow{\times X} R[X] \to R[X]/(X)
\]
in \(\mcalD(R[X])\), we obtain, after tensoring with $M$ over \(R[X]\), a fiber sequence
\begin{equation}\label{eq:derived-quotient}
   M \xrightarrow{\times f} M \to M/^L(f)
\end{equation}
in \(\mcalD(R)\).
The fiber sequence \eqref{eq:derived-quotient} yields an exact sequence
\[
0 \to \pi_{i}(M)/f\pi_{i}(M) \to \pi_{i}(M/^L(f)) \to \pi_{i-1}(M)[f] \to 0.
\]
Thus $M/^L(f)$ is connective and bounded since \(M\) is.
By \cref{extension-I-infty-torsion}, each $\pi_i(M/^L(f))$ is $I^\infty$-torsion, hence $M/^L(f)$ is cohomologically $I^\infty$-torsion.

For $r \geq 2$, note that
\[
M/^L(f_1,\ldots,f_r) \simeq (M/^L(f_1,\ldots,f_{r-1}))/^L(f_r).
\]
Set $N \defeq M/^L(f_1,\ldots,f_{r-1})$.
By the inductive hypothesis, $N$ is connective, bounded, and cohomologically $(f_1,\ldots,f_{r-1})^\infty$-torsion.
Applying the $r=1$ case to $N$ and $f_r$, we see that \(N/^L(f_r)\) is connective, bounded, and cohomologically \((f_r)^{\infty}\)-torsion.
Moreover, using the exact sequence
\[
0 \to \pi_{i}(N)/f_r\pi_{i}(N) \to \pi_{i}(N/^L(f_r)) \to \pi_{i-1}(N)[f_r] \to 0,
\]
and the fact that both $\pi_i(N)/f_r\pi_i(N)$ and $\pi_{i-1}(N)[f_r]$ are $I^\infty$-torsion, we deduce that
\[
\pi_{i}(M/^L(f_1,\ldots,f_r)) = \pi_i(N/^L(f_r))
\]
is $I^\infty$-torsion for all $i$.
Thus the claim follows.
\end{proof}

\begin{lemma} \label{BoundedProIsom}
    Let \(R\) be a ring and let \(M\) be an \(R\)-module.
    Take an element \(f\) of \(R\) and assume that \(M\) has bounded \(f^{\infty}\)-torsion.
    Then the morphism of pro-systems
    \begin{equation*}
        \{M/^Lf^n\}_{n \geq 1} \to \{M/f^nM\}_{n \geq 1}
    \end{equation*}
    induced from the canonical morphism \(M/^Lf^n \to M/f^nM\) becomes a pro-isomorphism.
\end{lemma}

\begin{proof}
    See the proof of \cite{bhatt2018Integral}*{Lemma 4.7(2)}.
\end{proof}

\subsection{Graded rings and coactions}

We recall the correspondence between graded rings and rings equipped with \(\setZ[G]\)-coactions (\Cref{corr-graded-coaction}).

\begin{definition}\label{def-coaction}
Let $R$ be a ring.
\begin{enumerate}
\item
Let
\[
R[G]\defeq \bigoplus_{g \in G} R \cdot t^g
\]
be the group ring with multiplication given by $t^g \cdot t^h=t^{g+h}$, namely, \(R[G]\) is the group ring of \(G\) over \(R\).
Then we have a natural isomorphism $R[G] \simeq \Z[G] \otimes_{\Z} R$.
Furthermore, $R[G]$ has the structure of a commutative and cocommutative Hopf \(R\)-algebra given by
\begin{align*}
\begin{alignedat}{2}
\Delta &\colon R[G] \to R[G] \otimes_R R[G]
&\qquad t^g &\mapsto t^g \otimes t^g, \\
e      &\colon R[G]\to R
&\qquad t^g   &\mapsto 1, \\
\iota  &\colon R[G] \to R[G]
&\qquad t^g &\mapsto t^{-g}.
\end{alignedat}
\end{align*}

\item
Let $R$ be a ring.
A \emph{$\Z[G]$-coaction} on $R$ is a ring homomorphism
\[
\rho \colon R \to R \otimes_{\Z} \Z[G]\simeq R[G]
\]
such that the following diagrams commute:
\begin{center}
\begin{minipage}{0.45\textwidth}
\[
\begin{tikzcd}
    R \arrow[r,"\rho"] \arrow[d,"\rho"']
        & R \otimes_{\Z} \Z[G] \arrow[d,"\id \otimes \Delta"] \\
    R \otimes_{\Z} \Z[G] \arrow[r,"\rho \otimes \id"]
        & R \otimes_{\Z} \Z[G] \otimes_{\Z} \Z[G]
\end{tikzcd}
\]
\end{minipage}
\hspace{0.5cm}
\begin{minipage}{0.35\textwidth}
\[
\begin{tikzcd}
    R \arrow[r,"\rho"] \arrow[d,equal]
        & R \otimes_{\Z} \Z[G] \arrow[ld,"\id \otimes e"] \\
    R.
\end{tikzcd}
\]
\end{minipage}
\end{center}
By abuse of notation, we denote by $e$ the ring homomorphism
\[
R[G] \simeq R \otimes_{\Z} \Z[G] \xrightarrow{\id \otimes e} R .
\]
\end{enumerate}
\end{definition}

\begin{proposition}\label{corr-graded-coaction}
The following categories are equivalent:
\begin{enumerate}
    \item the category $\mathcal{C}_1$ of $G$-graded rings, and
    \item the category $\mathcal{C}_2$ of rings equipped with $\Z[G]$-coactions.
\end{enumerate}
\end{proposition}

\begin{proof}
First, we construct a functor $\Phi \colon \mathcal{C}_1 \to \mathcal{C}_2$.
Let $R$ be a $G$-graded ring.
Then the ring homomorphism
\[
\rho \colon R \to R \otimes_{\Z} \Z[G],
\qquad
f=\sum_{g \in G} f_g \longmapsto \sum_{g \in G} f_g \otimes t^g,
\]
defines a coaction on $R$.
We set $\Phi(R) \defeq  (R,\rho)$.

Next, we construct a functor $\Psi \colon \mathcal{C}_2 \to \mathcal{C}_1$.
Let $R$ be a ring equipped with a coaction
\[
\rho \colon R \to R \otimes_{\Z} \Z[G].
\]
For each $g \in G$, we define an abelian subgroup $R_g$ of $R$ by
\[
R_g \defeq \rho^{-1}(R \otimes_{\Z} \Z \cdot t^g) \subseteq R.
\]
We define the $G$-graded structure on $R_{\gr}:=\bigoplus_{g \in G} R_g$.
For $g,h \in G$, we consider the homomorphism
\[
R_g \otimes_{\Z} R_h \to R \otimes_{\Z} R \xrightarrow{m} R \xrightarrow{\rho} R[G],
\]
where $m$ is the multiplication map on $R$.
The image of $r \otimes r'$ under the above map is
\[
\rho(rr')=\rho(r)\rho(r') \in (R\cdot t^g)(R \cdot t^h) \subseteq R \cdot t^{g+h}.
\]
Therefore, we obtain the multiplication
\[
R_g \otimes R_h \to R_{g+h},
\]
which defines the $G$-graded structure on $R_{\gr}$.

We then consider the natural ring homomorphism
\[
\varphi \colon R_{\gr} \defeq  \bigoplus_{g \in G} R_g \to R.
\]
We claim that $\varphi$ is an isomorphism.
If $\sum_g r_g=0$ with $r_g\in R_g$, counitality gives $\rho(r_g)=r_g\otimes t^g$, and applying $\rho$ shows that every $r_g$ is zero.
Thus $\varphi$ is injective.
Let $f \in R$, and write
\[
\rho(f)=\sum_{g \in G} f_g \otimes t^g.
\]
By the coassociativity of $\rho$ and the definition of $R_g$, each $f_g$ lies in $R_g$, and hence
\[
\rho\!\left(\sum_{g \in G} f_g\right)
= \sum_{g \in G} f_g \otimes t^g
= \rho(f).
\]
Since $e\circ \rho=\id$, it follows that
\[
f=\sum_{g \in G} f_g.
\]
Therefore, $\varphi$ is surjective, and $R$ admits a $G$-graded ring structure with $R=\bigoplus_{g \in G} R_g$.
We set $\Psi(R,\rho)\defeq \bigoplus_{g \in G} R_g$.

By construction, the functors $\Phi$ and $\Psi$ are quasi-inverse to each other.
\end{proof}

\section{Category of derived graded modules} \label{SectionDerivedGradedModules}

In this section, we define the \(\infty\)-category of derived \(G\)-graded modules, study its first properties, and introduce the notion of derived gradedwise completion (\Cref{DefDerivedGradedModules,PropertiesDerivedGradedModules,DerivedGradedwiseComp}).

\subsection{Definitions and properties}

We define the \(\infty\)-category of derived \(G\)-graded modules and establish its basic properties (\Cref{DefGradedModuleCat,PropertiesDgrZ,DefDerivedGradedModules,PropertiesDerivedGradedModules}).

\begin{notation} \label{DefGradedModuleCat}
    Let \(G^{ds}\) be the (ordinary) category of elements of \(G\) with trivial morphisms.
    We define the \emph{\(\infty\)-category of derived \(G\)-graded modules} \(\mcalD_{\graded{G}}(\setZ)\) as the \(\infty\)-category \(\Fun(G^{ds}, \mcalD(\setZ))\) of functors (\cite{lurie2009Higher}*{Notation 1.2.7.2}) from \(G^{ds}\) to the derived \(\infty\)-category \(\mcalD(\setZ)\) of \(\setZ\).
    For any \(M \in \mcalD_{\graded{G}}(\setZ)\) and an element \(g \in G\), we will denote the image of \(g\) by the functor \(M\) as \(M_g \in \mcalD(\setZ)\) and we call it the \emph{\(g\)-graded part} of \(M\).
    Based on the usual notation of graded modules, we sometimes denote \(M\) as \(\bigoplus_{g \in G} M_g\) and we can define the shift of \(M\) by \(M(g) \in \mcalD_{\graded{G}}(\setZ)\) where \(M(g)_h = M_{g+h}\) for any \(h \in G\).

    Since \(G^{ds}\) can be seen as a symmetric monoidal \(\infty\)-category by the additive structure \(+ \colon G^{ds} \times G^{ds} \to G^{ds}\), we can use \cite{lurie2017Higher}*{Example 2.2.6.17} to equip the \(\infty\)-category \(\mcalD_{\graded{G}}(\setZ)\) with a symmetric monoidal structure \(- \Lgrotimes_{\setZ} -\), which is called the \emph{Day convolution}: For any \(M\) and \(N\) in \(\mcalD_{\graded{G}}(\setZ) = \Fun(G^{ds}, \mcalD(\setZ))\), the Day convolution \(M \Lgrotimes_{\setZ} N\) is a left Kan extension of
    \begin{equation*}
        G^{ds} \times G^{ds} \xrightarrow{M \times N} \mcalD(\setZ) \times \mcalD(\setZ) \xrightarrow{\otimes^L_{\setZ}} \mcalD(\setZ)
    \end{equation*}
    along the additive structure \(+ \colon G^{ds} \times G^{ds} \to G^{ds}\).
    This is concretely given by
    \begin{equation*}
        M \Lgrotimes_{\setZ} N : G^{ds} \ni g \mapsto \colim_{s + t \to g} (M_s \otimes^L_{\setZ} N_t) = \bigoplus_{g = s+t} (M_s \otimes^L_{\setZ} N_t) \in \mcalD(\setZ)
    \end{equation*}
    for any \(g \in G\) since morphisms of \(G^{ds}\) are only trivial ones.
\end{notation}

\begin{remark}
    We can also define the derived \(\infty\)-category of \(G\)-graded abelian groups.
    In \Cref{SectionDerivedCat}, we will show that the derived \(\infty\)-category of the abelian category of \(G\)-graded abelian groups is equivalent to \(\mcalD_{\graded{G}}(\setZ)\).

    On the abelian category of \(G\)-graded abelian groups, we can also define a symmetric monoidal structure in the same way as above.
    We will write it by \(- \grotimes_{\setZ} -\).
    More generally, for a \(G\)-graded ring \(R\), we can define the abelian category of \(G\)-graded \(R\)-modules and its symmetric monoidal structure \(- \grotimes_{R} -\) in a usual way.
\end{remark}

\begin{lemma} \label{PropertiesDgrZ}
    On the symmetric monoidal \(\infty\)-category \(\mcalD_{\graded{G}}(\setZ)\), we have the following properties.
    \begin{enumerate}
        \item \label{PresentabilityDgrZ} The \(\infty\)-category \(\mcalD_{\graded{G}}(\setZ)\) is presentable and stable and is equivalent to the product \(\prod_{g \in G} \mcalD(\setZ)\) of \(\infty\)-categories (without the symmetric monoidal structure).
        \item \label{LimitsColimitsDgrZ} It has all small limits and small colimits. Moreover, any limits and colimits can be calculated gradedwise, i.e., the isomorphisms
        \begin{equation} \label{IsomColimGradedwise}
            (\grcolim_{i \in I} M_i)_g \cong  \colim_{i \in I} (M_i)_g \quad \text{and} \quad (\grlim_{i \in I} M_i)_g \cong  \lim_{i \in I} (M_i)_g
        \end{equation}
        hold in \(\mcalD(\setZ)\) for any \(g \in G\) and any simplicial sets \(I\) if these limits and colimits exist.
        \item \label{tStructureDgrZ} There exists a left and right complete accessible \(t\)-structure on \(\mcalD_{\graded{G}}(\setZ)\) such that \(M \in \mcalD_{\graded{G}}(\setZ)\) is connective (resp., coconnective) if and only if \(M_g \in \mcalD(\setZ)\) satisfies \(\pi_i(M_g) = 0\) for any \(i < 0\) (resp., \(i > 0\)) for all \(g \in G\).
        \item \label{ConnectivePartDgrZ} The connective part \(\mcalD_{\graded{G}, \geq 0}(\setZ)\) is stable under small colimits and extensions. The heart \(\mcalD_{\graded{G}}^{\heartsuit}(\setZ)\) is isomorphic to the abelian category \(\Mod_{\graded{G}}(\setZ)\) of \(G\)-graded abelian groups along \(M \mapsto \bigoplus_{g \in G} M_g\).
        \item \label{HomotopyGroupDgrZ} The homotopy group \(\pi_i(M)\) of \(M \in \mcalD_{\graded{G}}(\setZ)\) given by the \(t\)-structure is computed as the discrete \(G\)-graded abelian group \(\bigoplus_{g \in G}\pi_i(M_g)\) by using the homotopy group \(\pi_i(M_g)\) on \(\mcalD(\setZ)\).
        \item \label{ForgetfulFunctorDgrZ} Taking the coproduct, we can define a forgetful functor
        \begin{equation*}
            \mcalD_{\graded{G}}(\setZ) \to \mcalD(\setZ); \quad M = (M_g)_{g \in G} \mapsto \bigoplus_{g \in G} M_g.
        \end{equation*}
        This functor is conservative, symmetric monoidal, \(t\)-exact, and preserves all colimits.
        \item \label{GradedMappingSpectrumDgrZ} The symmetric monoidal structure on \(\mcalD_{\graded{G}}(\setZ)\) is closed, i.e., for any \(M \in \mcalD_{\graded{G}}(\setZ)\), the endofunctor \(- \Lgrotimes_{\setZ} M\) on \(\mcalD_{\graded{G}}(\setZ)\) admits a right adjoint, which we will write \(\grmapspt(M, -)\) and call it the \emph{graded mapping spectrum functor}.
    \end{enumerate}
\end{lemma}

\begin{proof}
    \Cref{PresentabilityDgrZ}: Since \(\mcalD(\setZ)\) is presentable (e.g., \cite{lurie2017Higher}*{Proposition 1.3.5.21}), then the presentability of \(\mcalD_{\graded{G}}(\setZ)\) follows from \cite{lurie2009Higher}*{Proposition 5.5.3.6}. The stability follows from the stability of \(\mcalD(\setZ)\) (\cite{lurie2017Higher}*{Proposition 1.3.5.9 and Proposition 1.1.3.1}).
    Since \(G^{ds}\) is a discrete category, the equivalence \(\mcalD_{\graded{G}}(\setZ) \simeq \prod_{g \in G} \mcalD(\setZ)\) follows from the definition of functor categories and products of categories \citeKero{0060} and \citeKero{00HW}.

    \Cref{LimitsColimitsDgrZ}: Since \(\mcalD_{\graded{G}}(\setZ)\) is presentable by \Cref{PresentabilityDgrZ}, it admits all small colimits and all small limits (\cite{lurie2009Higher}*{Definition 5.5.0.1 and Corollary 5.5.2.4}).
    The isomorphisms in \eqref{IsomColimGradedwise} follow from \cite{lurie2009Higher}*{Corollary 5.1.2.3} and its dual.

    \Cref{tStructureDgrZ}: For \(\mcalC \defeq \mcalD_{\graded{G}}(\setZ)\), we define the full subcategory \(\mcalC_{\geq 0}\) (resp., \(\mcalC_{\leq 0}\)) of \(\mcalC\) spanned by objects \(M\) such that \(M_g \in \mcalD(\setZ)_{\geq 0}\) (resp., \(M_g \in \mcalD(\setZ)_{\leq 0}\)) for any \(g \in G\).
    We will check that this pair \((\mcalC_{\geq 0}, \mcalC_{\leq 0})\) forms an accessible t-structure on \(\mcalC\).
    The accessibility means that \(\mcalC_{\geq 0}\) is presentable (\cite{lurie2017Higher}*{Definition 1.4.4.12}), which follows from the equivalence
    \begin{equation*}
        \mcalC_{\geq 0} \simeq \prod_{g \in G} \mcalD_{\geq 0}(\setZ)
    \end{equation*}
    of \(\infty\)-categories by \Cref{PresentabilityDgrZ} and the accessibility of the standard \(t\)-structure on \(\mcalD(\setZ)\) (\cite{lurie2017Higher}*{Proposition 7.1.1.13}).
    We will show that this pair defines a \(t\)-structure on \(\mcalC\).
    By \Cref{LimitsColimitsDgrZ}, we know that the \(n\)-shift \(M[n]\) in \(\mcalC\), which is now given by a pullback diagram in \(\mcalC\), is \(\bigoplus_{g \in G} M_g[n]\), where \(M_g[n] \in \mcalD(\setZ)\) is the \(n\)-shift of \(M_g\) in \(\mcalD(\setZ)\).
    Since we use the standard \(t\)-structure on \(\mcalD(\setZ)\), first we have \(\mcalC_{\geq 0}[1] \subseteq \mcalC_{\geq 0}\) and \(\mcalC_{\leq 0}[-1] \subseteq \mcalC_{\leq 0}\).
    Composing the truncation functors \(\tau_{\geq n}\) and \(\tau_{\leq n}\) on \(\mcalD(\setZ)\), we can define functors\footnote{The notation \(\tau^{\graded}_{\gtreqqless}\) used here is temporary and will not be used later. This is just to distinguish the truncation functors on \(\mcalD(\setZ)\) and we will use only \(\tau_{\gtreqqless}\) for the truncation functors on any stable \(\infty\)-category with a \(t\)-structure.}
    \begin{align*}
        \tau^{\graded}_{\geq n} \colon \mcalC \to \mcalC_{\geq n}; & \quad M = (M_g)_{g \in G} \mapsto (\tau_{\geq n}M_g)_{g \in G} \\
        \tau^{\graded}_{\leq n} \colon \mcalC \to \mcalC_{\leq n}; & \quad M = (M_g)_{g \in G} \mapsto (\tau_{\leq n}M_g)_{g \in G}
    \end{align*}
    such that any \(M \in \mcalC\) has a sequence
    \begin{equation*}
        \tau^{\graded}_{\geq 0}M \to M \to \tau^{\graded}_{\leq -1}M
    \end{equation*}
    in \(\mcalC\). Since this sequence gives a fiber sequence on each graded part, this is a fiber sequence in \(\mcalC\) by \Cref{LimitsColimitsDgrZ}.
    Finally, for any \(M \in \mcalC_{\geq 0}\) and \(N \in \mcalC_{\leq -1}\), we have
    \begin{equation} \label{MappingSpaceDecomp}
        \Map_{\mcalC}(M, N) \simeq \prod_{g \in G} \Map_{\mcalD(\setZ)}(M_g, N_g) \simeq 0,
    \end{equation}
    where the first equivalence follows from \Cref{PresentabilityDgrZ} and \citeKero{034Y} and \(\Map_{\mcalD(\setZ)}(M_g, N_g) \cong 0\) holds for any \(g \in G\) by the definition of the \(t\)-structures on \(\mcalD(\setZ)\) and \(\mcalC\).
    These three properties show that \((\mcalC_{\geq 0}, \mcalC_{\leq 0})\) is a \(t\)-structure on \(\mcalC\).
    Its truncations and limits and colimits are computed degreewise, so the completeness of the standard \(t\)-structure on \(\mcalD(\setZ)\) gives
    \begin{equation*}
        M\xrightarrow{\simeq}\grlim_{n\geq0}\tau_{\leq n}M, \qquad \grcolim_{n\geq0}\tau_{\geq-n}M\xrightarrow{\simeq}M.
    \end{equation*}

    \Cref{ConnectivePartDgrZ}: Using \Cref{LimitsColimitsDgrZ}, the connective part \(\mcalD_{\graded{G}, \geq 0}(\setZ)\) is stable under small colimits and extensions since the same property holds on \(\mcalD(\setZ)\) by \cite{lurie2017Higher}*{Proposition 7.1.1.13}.
    The heart \(\mcalD_{\graded{G}}^{\heartsuit}(\setZ)\) consists of \(M \in \mcalD_{\graded{G}}(\setZ)\) such that \(M_g \in \mcalD^{\heartsuit}(\setZ) \simeq \Ab\) for any \(g \in G\).
    Therefore, \(M\) can be seen as a \(G\)-graded abelian group \(\bigoplus_{g \in G} M_g\).

    \Cref{HomotopyGroupDgrZ}: The homotopy group \(\pi_i(M)\) of \(M \in \mcalD_{\graded{G}}(\setZ)\) belongs to the heart \(\mcalD_{\graded{G}}^{\heartsuit}(\setZ)\) and thus it is a \(G\)-graded abelian group.
    This is defined as \(\tau^{\graded}_{\geq 0}\tau^{\graded}_{\leq 0}(M[-i])\) and thus its \(g\)-graded part is \(\tau_{\geq 0}\tau_{\leq 0}(M_g[-i]) \cong \pi_i(M_g)\) for any \(g \in G\) by the definition of \(\tau^{\graded}_{\gtreqqless 0}\).

    \Cref{ForgetfulFunctorDgrZ}: This functor is the same as the functor which sends \(G^{ds}\)-indexed diagram \((M_g)_{g \in G}\) in \(\mcalD(\setZ)\) to its coproduct \(\bigoplus_{g \in G} M_g\) in \(\mcalD(\setZ)\).
    The conservativity is clear since if \(\bigoplus_{g \in G} M_g \cong 0\) in \(\mcalD(\setZ)\), then \(M_g \cong 0\) in \(\mcalD(\setZ)\) for any \(g \in G\) and thus \(M \cong 0\) in \(\mcalD_{\graded{G}}(\setZ)\).
    Also, taking the coproduct preserves all colimits and then \Cref{LimitsColimitsDgrZ} ensures the preservation of colimits.
    The \(t\)-exactness follows from \Cref{tStructureDgrZ}.
    On the symmetric monoidal structure, the symmetric product \(M \Lgrotimes_{\setZ} N\) in \(\mcalD_{\graded{G}}(\setZ)\) for any \(M\) and \(N\) in \(\mcalD_{\graded{G}}(\setZ)\) goes to
    \begin{equation*}
        M \Lgrotimes_{\setZ} N \mapsto \bigoplus_{g \in G} \parenlr{\bigoplus_{g = s + t} (M_s \otimes^L_{\setZ} N_t)} \cong \parenlr{\bigoplus_{s \in G} M_s} \otimes^L_{\setZ} \parenlr{\bigoplus_{t \in G} N_t}
    \end{equation*}
    in \(\mcalD(\setZ)\) for any small index set \(I\). This finishes the proof.

    \Cref{GradedMappingSpectrumDgrZ}: Since \(\mcalD_{\graded{G}}(\setZ)\) is presentable, we can use the adjoint functor theorem (\cite{lurie2009Higher}*{Corollary 5.5.2.9}) and then it is enough to show that the endofunctor \(- \Lgrotimes_{\setZ} M\) preserves all small colimits, i.e., the canonical morphism \(\grcolim_{i \in I} (N_i \Lgrotimes_{\setZ} M) \to (\grcolim_{i \in I} N_i) \Lgrotimes_{\setZ} M\) is an isomorphism in \(\mcalD_{\graded{G}}(\setZ)\) for any small colimit diagram \(i \mapsto N_i\) in \(\mcalD_{\graded{G}}(\setZ)\).
    Using the conservativity in \Cref{ForgetfulFunctorDgrZ} above, this can be reduced to the same problem on the endofunctor \(- \otimes^L_{\setZ} M\) on \(\mcalD(\setZ)\). This is a usual property so we are done.
    \qedhere
\end{proof}

\begin{remark} \label{NonIsomGradedLimit}
    Since the forgetful functor in \Cref{PropertiesDgrZ}\Cref{ForgetfulFunctorDgrZ} preserves all colimits, it is exact and especially it preserves finite limits (\cite{lurie2017Higher}*{Proposition 1.1.4.1}).
    However, this does not preserve (infinite) limits in general.
    For example, take \(G = \setZ\) and consider the inverse system \(F \colon \setZ_{\geq 0}^{\opposite} \to \mcalD_{\graded{G}}(\setZ)\) given by \(F(n) = \setZ[X]/(X^n)\) with \(\deg(X) = 1\) and the natural projection morphisms.
    Then, we can show that
    \begin{equation*}
        \grlim_{n \geq 0} \setZ[X]/(X^n) \cong \setZ[X] \quad \text{in } \mcalD_{\graded{G}}(\setZ)
    \end{equation*}
    but the usual limit \(\lim_{n \geq 0} \setZ[X]/(X^n)\) in \(\mcalD(\setZ)\) is isomorphic to the ring of formal power series \(\setZ[|X|]\).
\end{remark}

\begin{definition} \label{DefDerivedGradedModules}
    Let \(R\) be a (discrete) \(G\)-graded ring which we can regard as an \(\setE_{\infty}\)-ring object in the symmetric monoidal \(\infty\)-category \(\mcalD_{\graded{G}}(\setZ)\) by setting \(R_g\) as a discrete object of \(\mcalD(\setZ)\) for any \(g \in G\).

    Then the \emph{\(\infty\)-category \(\mcalD_{\graded{G}}(R)\) of derived \(G\)-graded \(R\)-modules} is the \(\infty\)-category \(\Mod_{R}(\mcalD_{\graded{G}}(\setZ))\) of \(R\)-modules in the symmetric monoidal \(\infty\)-category \(\mcalD_{\graded{G}}(\setZ)\).
    For \(d\in G\), we define the graded $\Z$-module $\Z(d)$ by
    \[
    \Z(d)_{m}=
    \begin{cases}
        \Z & m=d \\
        0 & m \neq d
    \end{cases}
    \]
    and the grade shift by the functor \(M(d)\defeq M\Lgrotimes_{\setZ}\setZ(d)\); thus \(M(d)_g=M_{g+d}\).
    Similarly, we define \emph{the \(\infty\)-category \(\CAlg(\mcalD_{\graded{G}}(R))\) of \(G\)-graded \(R\)-algebras} as the \(\infty\)-category \(\CAlg_{R}(\mcalD_{\graded{G}}(\setZ))\) of \(R\)-algebras in the symmetric monoidal \(\infty\)-category \(\mcalD_{\graded{G}}(\setZ)\).

    This notation is unambiguous in view of \Cref{DefGradedModuleCat}, since if \(R=\setZ\), then
\[
\mcalD_{\graded{G}}(R)=\Mod_{\setZ}(\mcalD_{\graded{G}}(\setZ)) \simeq \mcalD_{\graded{G}}(\setZ)
\]
by \cite{lurie2017Higher}*{Proposition 4.2.4.9}.
\end{definition}

\begin{proposition} \label{PropertiesDerivedGradedModules}
    On \(\mcalD_{\graded{G}}(R)\), we have the following properties.
    \begin{enumerate}
        \item \label{PresentabilityDgrR} The \(\infty\)-category \(\mcalD_{\graded{G}}(R)\) is presentable and stable.
        \item \label{LimitsColimitsDgrR} The canonical functor \(\mcalD_{\graded{G}}(R) \to \mcalD_{\graded{G}}(\setZ)\) induced from \(\setZ \to R\) is conservative and preserves all limits and colimits. In particular, its left adjoint is given by the base change \(- \Lgrotimes_{\setZ} R \colon \mcalD_{\graded{G}}(\setZ) \to \mcalD_{\graded{G}}(R)\) used by the symmetric monoidal structure on \(\mcalD_{\graded{G}}(\setZ)\).
        \item \label{ClosedMonoidalDgrR} The symmetric monoidal structure on \(\mcalD_{\graded{G}}(\setZ)\) induces a symmetric monoidal structure \(- \Lgrotimes_R -\) on \(\mcalD_{\graded{G}}(R)\) and it is closed, i.e., for any \(M \in \mcalD_{\graded{G}}(R)\), the endofunctor \(- \Lgrotimes_R M\) on \(\mcalD_{\graded{G}}(R)\) admits a right adjoint \(\grmapspt_R(M, -)\).
        \item \label{tStructureDgrR} Using the \(t\)-structure on \(\mcalD_{\graded{G}}(\setZ)\) in \Cref{PropertiesDgrZ}\Cref{tStructureDgrZ}, we can equip \(\mcalD_{\graded{G}}(R)\) with a left and right complete accessible \(t\)-structure such that \(M \in \mcalD_{\graded{G}}(R)\) is connective (resp., coconnective) if and only if \(M_g \in \mcalD(\setZ)\) satisfies \(\pi_i(M_g) = 0\) for any \(i < 0\) (resp., \(i > 0\)) for all \(g \in G\).
        \item \label{ConnectivePartDgrR} The connective part \(\mcalD_{\graded{G}, \geq 0}(R)\) is stable under small colimits and extensions. The heart \(\mcalD_{\graded{G}}^{\heartsuit}(R)\) of this \(t\)-structure is isomorphic to the abelian category \(\Mod_{\graded{G}}(R)\) of (discrete) \(G\)-graded \(R\)-modules.
        \item \label{HomotopyGroupsDgrR} The \(t\)-structure inherits the homotopy groups \(\pi_i(M)\) of \(M \in \mcalD_{\graded{G}}(R)\) whose underlying \(G\)-graded abelian group is \(\bigoplus_{g \in G} \pi_i(M_g)\) by using the homotopy group \(\pi_i(M_g)\) on \(\mcalD(\setZ)\).
        \item \label{ForgetfulFunctorDgrR} The functor \(\mcalD_{\graded{G}}(\setZ) \to \mcalD(\setZ)\) in \Cref{PropertiesDgrZ}\Cref{ForgetfulFunctorDgrZ} induces a forgetful functor
        \begin{equation*}
            \theta = \theta_R \colon \mcalD_{\graded{G}}(R) \to \mcalD(R); \quad M = (M_g)_{g \in G} \mapsto \bigoplus_{g \in G} M_g,
        \end{equation*}
        which is conservative, symmetric monoidal, \(t\)-exact, and preserves all colimits.
    \end{enumerate}
\end{proposition}

\begin{proof}
    \Cref{PresentabilityDgrR}: Since \(\mcalD_{\graded{G}}(\setZ)\) is presentable by \Cref{PropertiesDgrZ}\Cref{PresentabilityDgrZ} and the functor \(- \otimes_{\setZ} M\) on \(\mcalD_{\graded{G}}(\setZ)\) preserves all small colimits for any \(M \in \mcalD_{\graded{G}}(\setZ)\) by \Cref{PropertiesDgrZ}\Cref{GradedMappingSpectrumDgrZ}, the presentability of \(\mcalD_{\graded{G}}(R)\) follows from \cite{lurie2017Higher}*{Corollary 4.2.3.7}.
    The stability follows from \cite{lurie2017Higher}*{Proposition 4.2.3.4 and Proposition 7.1.1.4}.

    \Cref{LimitsColimitsDgrR}: First note that \(\setZ \to R\) in \(\CAlg(\mcalD_{\graded{G}}(\setZ))\) induces a functor \(\Mod_R(\mcalD_{\graded{G}}(\setZ)) \to \Mod_{\setZ}(\mcalD_{\graded{G}}(\setZ))\), which gives the functor \(\mcalD_{\graded{G}}(R) \to \mcalD_{\graded{G}}(\setZ)\) by \Cref{DefDerivedGradedModules}.
    Preserving limits (resp., colimits) follows from \cite{lurie2017Higher}*{Corollary 4.2.3.3(2) (resp., Corollary 4.2.3.5(2))} since the symmetric product on \(\mcalD_{\graded{G}}(\setZ)\) commutes with colimits by \Cref{PropertiesDgrZ}\Cref{GradedMappingSpectrumDgrZ}.
    The description of a left adjoint is given by \cite{lurie2017Higher}*{Remark 4.5.3.3}.

    \Cref{ClosedMonoidalDgrR}: Using \Cref{PropertiesDgrZ}\Cref{GradedMappingSpectrumDgrZ}, the symmetric product on \(\mcalD_{\graded{G}}(\setZ)\) preserves all small colimits in each variable. Then \cite{lurie2017Higher}*{Definition 4.5.1.1, Theorem 4.5.2.1, and Proposition 4.4.3.12} shows that we can define a symmetric monoidal structure \(- \Lgrotimes_R -\) on \(\mcalD_{\graded{G}}(R)\) by the Bar construction.
    For the closedness, we can use the adjoint functor theorem as in \Cref{PropertiesDgrZ}\Cref{GradedMappingSpectrumDgrZ}. Actually, \(\mcalD_{\graded{G}}(R)\) is presentable by \Cref{PresentabilityDgrR} and the relative tensor product \(- \Lgrotimes_R -\) preserves all small colimits by \cite{lurie2017Higher}*{Corollary 4.4.2.15}.

    \Cref{tStructureDgrR}: We will imitate the proof of \cite{lurie2018Spectral}*{Proposition 2.1.1.1}.
    For simplicity, we set \(\mcalC_R \defeq \mcalD_{\graded{G}}(R)\) and \(\mcalC_{\setZ} \defeq \mcalD_{\graded{G}}(\setZ)\) and write the functor in \Cref{LimitsColimitsDgrR} as \(\tau \colon \mcalC_R \to \mcalC_{\setZ}\).
    By \Cref{LimitsColimitsDgrR} and the \(t\)-structure on \(\mcalC_{\Z}\) by \Cref{PropertiesDgrZ}\Cref{tStructureDgrZ}, the full subcategory \(\mcalC_{R, \geq 0}\) of \(\mcalC_R\) spanned by objects \(M\) such that \(\tau(M) \in \mcalC_{\setZ, \geq 0}\) is closed under small colimits and extensions in \(\mcalC_R\).
    Since \(R\) and the tensor product of connective objects are connective, \(\mcalC_{R,\geq0}\) identifies with \(\Mod_R(\mcalC_{\setZ,\geq0})\), which is presentable by \cite{lurie2017Higher}*{Corollary 4.2.3.7}.
    Hence \cite{lurie2017Higher}*{Proposition 1.4.4.11} gives an accessible \(t\)-structure \((\mcalC_{R,\geq0},\mcalC_R')\) on \(\mcalC_R\).
    It suffices to show that an object \(M\) belongs to \(\mcalC_R'\) if and only if \(\tau(M)\) belongs to \(\mcalC_{\setZ, \leq 0}\).

    (Only if part): Assume that \(M \in \mcalC_{R}'\) and take any \(N \in \mcalC_{\setZ, \geq 0}\).
    It suffices to show that \(\Map_{\mcalC_{\setZ}}(N, \tau(M)[-1]) \simeq 0\) holds.
    Since \Cref{LimitsColimitsDgrR} shows that a left adjoint functor \(- \Lgrotimes_{\setZ} R \colon \mcalC_{\setZ} \to \mcalC_R\) of \(\tau\) is given by the symmetric monoidal structure on \(\mcalC_{\setZ}\), the definition of the Day convolution in \Cref{DefGradedModuleCat} says that \(\tau(N \Lgrotimes_{\setZ} R) \in \mcalC_{\setZ}\) is connective.
    Therefore, we have \(N \Lgrotimes_{\setZ} R \in \mcalC_{R, \geq 0}\) by the definition and then the equivalences
    \begin{equation*}
        \Map_{\mcalC_{\setZ}}(N, \tau(M)[-1]) \simeq \Map_{\mcalC_R}(N \Lgrotimes_{\setZ} R, M[-1]) \simeq 0
    \end{equation*}
    holds since we already know that the pair \((\mcalC_{R, \geq 0}, \mcalC_R')\) is a \(t\)-structure.
    This shows that \(\tau(M)\) belongs to \(\mcalC_{\setZ, \leq 0}\).

    (If part): It suffices to show that \(\Map_{\mcalC_R}(N, M[-1]) \simeq 0\) holds for any \(N \in \mcalC_{R, \geq 0}\).
    Take a full subcategory \(\mcalM\) of \(\mcalC_R\) consisting of objects \(N \in \mcalC_R\) such that \(\Map_{\mcalC_R}(N, M[-1]) \simeq 0\).
    We want to show that \(\mcalM\) contains \(\mcalC_{R, \geq 0}\).
    Since the functors \(\tau \colon \mcalC_R \to \mcalC_{\setZ}\) and \(\mcalC_{\setZ, \geq 0} \hookrightarrow \mcalC_{\setZ}\) are conservative and preserve all small colimits by \Cref{LimitsColimitsDgrR} and \Cref{PropertiesDgrZ}\Cref{ConnectivePartDgrZ}, they induce a functor
    \begin{equation*}
        \mcalD_{\graded{G}, \geq 0}(R) = \mcalC_{R, \geq 0} \xrightarrow{\tau} \mcalC_{\setZ, \geq 0} = \mcalD_{\graded{G}, \geq 0}(\setZ)
    \end{equation*}
    which is again conservative and preserves all small colimits and has a left adjoint \(- \Lgrotimes_{\setZ} R\).
    Using \cite{lurie2017Higher}*{Proposition 4.7.3.14}, any object \(N \in \mcalC_{R, \geq 0}\) can be written as a simplicial colimit of objects contained in the essential image of the left adjoint \(- \Lgrotimes_{\setZ} R\) of \(\tau\).
    The full subcategory \(\mcalM\) is stable under small colimits and then it suffices to show that it contains the essential image of \(- \Lgrotimes_{\setZ} R \colon \mcalC_{\setZ, \geq 0} \to \mcalC_{R, \geq 0}\).
    Take any \(N' \in \mcalC_{\setZ, \geq 0}\).
    The adjunction says that
    \begin{equation*}
        \Map_{\mcalC_R}(N' \Lgrotimes_{\setZ} R, M[-1]) \simeq \Map_{\mcalC_{\setZ}}(N', \tau(M)[-1]) \simeq 0
    \end{equation*}
    because of the assumption \(\tau(M) \in \mcalC_{\setZ, \leq 0}\).

    Consequently, \(\mcalD_{\graded{G}}(R)\) has an accessible \(t\)-structure such that the conservative functor \(\mcalD_{\graded{G}}(R) \to \mcalD_{\graded{G}}(\setZ)\) is \(t\)-exact.
    The functor \(\tau\) commutes with truncations and preserves limits and colimits.
    It therefore carries the canonical morphisms
    \begin{equation*}
        M\to\grlim_{n\geq0}\tau_{\leq n}M, \qquad \grcolim_{n\geq0}\tau_{\geq-n}M\to M
    \end{equation*}
    to isomorphisms in \(\mcalC_{\setZ}\).
    Its conservativity proves left and right completeness.

    \Cref{ConnectivePartDgrR}: The first assertion follows from the above proof of \Cref{tStructureDgrR}. The second one follows from that \(M \in \mcalD_{\graded{G}}(R)\) belongs to the heart if and only if \(M\) belongs to \(\mcalD_{\graded{G}}^{\heartsuit}(\setZ)\) with an \(R\)-module structure.
    This is equivalent to saying that \(M\) is a discrete \(G\)-graded \(R\)-module by \Cref{PropertiesDgrZ}\Cref{ConnectivePartDgrZ}.

    \Cref{HomotopyGroupsDgrR}: Since the functor \(\mcalD_{\graded{G}}(R) \to \mcalD_{\graded{G}}(\setZ)\) is \(t\)-exact by \Cref{tStructureDgrR}, the underlying \(G\)-graded abelian group of \(\pi_i(M)\) is given by \(\bigoplus_{g \in G} \pi_i(M_g)\).

    \Cref{ForgetfulFunctorDgrR}: Using \cite{lurie2017Higher}*{Remark 4.8.3.25}, the functor in \Cref{PropertiesDgrZ}\Cref{ForgetfulFunctorDgrZ} induces a commutative diagram
    \begin{center}
        \begin{tikzcd}
            \mcalD_{\graded{G}}(R) \arrow[r] \arrow[d] & \mcalD(R) \arrow[d] \\
            \mcalD_{\graded{G}}(\setZ) \arrow[r]       & \mcalD(\setZ)
        \end{tikzcd}
    \end{center}
    where the vertical functors are conservative, \(t\)-exact, and preserve limits and colimits by \Cref{LimitsColimitsDgrR} and \Cref{tStructureDgrR}.
    On the other hand, the horizontal functors are symmetric monoidal by the construction and the lower one is conservative, symmetric monoidal, \(t\)-exact, and preserves all colimits by \Cref{PropertiesDgrZ}\Cref{ForgetfulFunctorDgrZ}.
    Combining these results, we can conclude the desired properties of the upper horizontal functor.
\end{proof}

\subsection{Graded Milnor exact sequences}

We prove a graded Milnor exact sequence and its Mittag--Leffler consequence (\Cref{MilnorExactSequence,MittagLefflerDiscrete}).

\begin{lemma} \label{MilnorExactSequence}
    Let \((M_n, f_{n+1} \colon M_{n+1} \to M_n)_{n \geq 0}\) be a \(\setZ_{\geq 0}\)-indexed diagram in \(\mcalD_{\graded{G}}(R)\).
    Then we have a fiber sequence
    \begin{equation*}
        \grlim_{n \geq 0} M_n \to \prod^{\graded}_{n \geq 0} M_n \xrightarrow{\id - s} \prod^{\graded}_{n \geq 0} M_n
    \end{equation*}
    in \(\mcalD_{\graded{G}}(R)\), where \(\prod^{\graded}\) is the product in \(\mcalD_{\graded{G}}(R)\) and the morphism \(\id - s \colon \prod^{\graded}_{n \geq 0} M_n \to \prod^{\graded}_{n \geq 0} M_n\) is given by \((m_n)_{n \geq 0} \mapsto (m_n - f_{n+1}(m_{n+1}))_{n \geq 0}\).
    Especially, the following exact sequence of (discrete) \(G\)-graded \(R\)-modules, which is a graded variant of the \emph{Milnor exact sequence},
    \begin{equation*}
        0 \to R^1\!\grlim_{n \geq 0} \pi_{i+1}(M_n) \to \pi_i(\grlim_{n \geq 0} M_n) \to R^0\!\grlim_{n \geq 0} \pi_i(M_n) \to 0
    \end{equation*}
    exists for any \(i \in \setZ\), where \(R^i\grlim(-) = \pi_{-i}(\grlim(-))\).
\end{lemma}

\begin{proof}
    Since all small limits can be computed gradedwise (\Cref{PropertiesDerivedGradedModules}\Cref{LimitsColimitsDgrR} and \Cref{PropertiesDgrZ}\Cref{LimitsColimitsDgrZ}), it suffices to show that the same fiber sequence exists in \(\mcalD(\setZ)\) for the diagram \((M_{n, g}, f_{n+1, g} \colon M_{n+1, g} \to M_{n, g})_{n \geq 0}\) for any \(g \in G\).
    This existence is a usual result on derived limits and homotopy limits in \(\mcalD(\setZ)\).

    Taking the homotopy long exact sequence of the fiber sequence, we have a long exact sequence
    \begin{equation}\label{eq:Milnor-exact}
        \prod^{\graded}_{n \geq 0} \pi_{i+1}(M_n) \xrightarrow{\id - s} \prod^{\graded}_{n \geq 0} \pi_{i+1}(M_n) \to \pi_i(\grlim_{n \geq 0} M_n) \to \prod^{\graded}_{n \geq 0} \pi_i(M_n) \xrightarrow{\id - s} \prod^{\graded}_{n \geq 0} \pi_i(M_n)
    \end{equation}
    for any \(i \in \setZ\) in the abelian category of \(G\)-graded \(R\)-modules.
    After replacing \((M_n, f_{n+1})_{n \geq 0}\) with $(\pi_i(M_n),\pi_i(f_{n+1}))_{n \geq 0}$ and $(\pi_{i+1}(M_n),\pi_{i+1}(f_{n+1}))_{n \geq 0}$ and then reapplying \eqref{eq:Milnor-exact}, we obtain
    \begin{align*}
        &\Ker (\prod^{\graded}_{n \geq 0} \pi_{i}(M_n) \xrightarrow{\id - s} \prod^{\graded}_{n \geq 0} \pi_{i}(M_n)) = \pi_{0}(\grlim_{n \geq 0} \pi_{i}(M_n))\\
        &\Coker(\prod^{\graded}_{n \geq 0} \pi_{i+1}(M_n) \xrightarrow{\id - s} \prod^{\graded}_{n \geq 0} \pi_{i+1}(M_n))=\pi_{-1}(\grlim_{n \geq 0} \pi_{i+1}(M_n)).
    \end{align*}
    Therefore, combining \eqref{eq:Milnor-exact} for the original \((M_n, f_{n+1})_{n \geq 0}\), we obtain
    \begin{equation*}
        0 \to \pi_{-1}(\grlim_{n \geq 0} \pi_{i+1}(M_n)) \to \pi_i(\grlim_{n \geq 0} M_n) \to \pi_0(\grlim_{n \geq 0} \pi_i(M_n)) \to 0.
    \end{equation*}
    This is the exact sequence we want.
\end{proof}

\begin{corollary} \label{MittagLefflerDiscrete}
    Let \((M_n, f_{n+1} \colon M_{n+1} \to M_n)_{n \geq 0}\) be a \(\setZ_{\geq 0}\)-indexed diagram of discrete \(G\)-graded \(R\)-modules.
    If each \(f_n\) is surjective, then the limit \(\grlim_{n \geq 0} M_n\) in \(\mcalD_{\graded{G}}(R)\) is concentrated in degree \(0\) and computes the limit of \((M_n)_{n \geq 0}\) in the abelian category \(\Mod_{\graded{G}}(R)\) of discrete \(G\)-graded \(R\)-modules.
\end{corollary}

\begin{proof}
    The Milnor exact sequence (\Cref{MilnorExactSequence}) shows that \(\grlim_{n \geq 0}M_n\) is concentrated in (homological) degree \([-1, 0]\).
    On \(i = -1\), we have an isomorphism \(\pi_{-1}(\grlim_{n \geq 0}M_n) \cong R^1\grlim_{n \geq 0}M_n\) of discrete \(G\)-graded \(R\)-modules.
    Using the surjectivity of \(f_n\)'s, the morphism \(\id - s \colon \prod^{\graded}_{n \geq 0} M_n \to \prod^{\graded}_{n \geq 0} M_n\) in \Cref{MilnorExactSequence} becomes surjective.
    Then \eqref{eq:Milnor-exact} shows the vanishing of \(R^1\grlim_{n \geq 0}M_n\).
    This shows the discreteness of \(\grlim_{n \geq 0}M_n\).

    Take any diagram \((p_n \colon N \to M_n)_{n \geq 0}\) of discrete \(G\)-graded \(R\)-modules which is compatible with \(f_n\)'s.
    Sending this to \(\mcalD_{\graded{G}}(R)\), this diagram corresponds to the unique morphism \(N \to \grlim_{n \geq 0}M_n\) in \(\mcalD_{\graded{G}}(R)\).
    Since \(\grlim_{n \geq 0}M_n\) and \(N\) are concentrated in degree \(0\), the adjointness of the truncation functors on \(\mcalC \defeq \mcalD_{\graded{G}}(R)\) shows that
    \begin{align*}
        \Map_{\mcalC}(N, \grlim_{n \geq 0}M_n) & \simeq \Map_{\mcalC_{\geq 0}}(N, \tau_{\geq 0}(\grlim_{n \geq 0}M_n)) \\
        & \simeq \Map_{\mcalC^{\heartsuit}}(\tau_{\leq 0}(N), \grlim_{n \geq 0}M_n) \simeq \Map_{\mcalC^{\heartsuit}}(N, \grlim_{n \geq 0}M_n).
    \end{align*}
    Then the morphism \(N \to \grlim_{n \geq 0}M_n\) in \(\mcalD_{\graded{G}}(R)\) corresponds to the unique morphism \(N \to \grlim_{n \geq 0}M_n\) in the heart \(\mcalC^{\heartsuit} \simeq \Mod_{\graded{G}}(R)\) which is compatible with the given diagram \((p_n)\).
    This shows that \(\grlim_{n \geq 0}M_n\) satisfies the universality of the limit of \((M_n)_{n \geq 0}\) in \(\Mod_{\graded{G}}(R)\).
\end{proof}

\subsection{Derived category of graded modules and compact generators} \label{SectionDerivedCat}

This subsection compares the derived \(\infty\)-category \(\mcalD(\Mod_{\graded{G}}(R))\) of the abelian category of \(G\)-graded \(R\)-modules with the \(\infty\)-category \(\mcalD_{\graded{G}}(R)\) of derived \(G\)-graded \(R\)-modules defined in \Cref{DefDerivedGradedModules}.

\begin{definition} \label{DefDerivedGradedCat}
    Let \(\Mod_{\graded{G}}(R)\) be the abelian category of \(G\)-graded \(R\)-modules.
    Since this is a Grothendieck abelian category with a generator \(\bigoplus_{g \in G} R(g)\), we can define the \emph{derived \(\infty\)-category \(\mcalD(\Mod_{\graded{G}}(R))\) of \(\Mod_{\graded{G}}(R)\)} by \cite{lurie2017Higher}*{Definition 1.3.5.8}.
    This is a presentable stable \(\infty\)-category with a right complete and accessible \(t\)-structure by \cite{lurie2017Higher}*{Proposition 1.3.5.9 and Proposition 1.3.5.21}.

    Using the forgetful functor \(\Mod_{\graded{G}}(R) \to \Mod(R)\) to the abelian category of \(R\)-modules, we can define the forgetful functor
    \begin{equation*}
        \mcalD(\Mod_{\graded{G}}(R)) \to \mcalD(\Mod(R))
    \end{equation*}
    of stable \(\infty\)-categories.
\end{definition}

\begin{proposition} \label{DerivedGradedCatEquiv}
    There exists a \(t\)-exact equivalence of stable \(\infty\)-categories
    \begin{equation*}
        \eta \colon \mcalD(\Mod_{\graded{G}}(R)) \xrightarrow{\simeq} \mcalD_{\graded{G}}(R),
    \end{equation*}
    which makes the following diagram commutative:
    \begin{center}
        \begin{tikzcd}
            \mcalD(\Mod_{\graded{G}}(R)) \arrow[d] \arrow[r, "\eta"] & \mcalD_{\graded{G}}(R) \arrow[d] \\
            \mcalD(\Mod(R)) \arrow[r, "\simeq"]                           & \mcalD(R)
        \end{tikzcd}
    \end{center}
    where the vertical functors are forgetful functors and the lower equivalence is given in \cite{lurie2017Higher}*{Proposition 7.1.1.15 and Remark 7.1.1.16}.
\end{proposition}

\begin{proof}
    The abelian category \(\Mod_{\graded{G}}(R)\) has a class of projective objects given by direct sums of grade shifts of free \(R\)-modules, which is called graded-free \(R\)-modules.
    Since graded-free \(R\)-modules give enough projective objects of \(\Mod_{\graded{G}}(R)\), we can take the full subcategory \(\mcalD^-(\Mod_{\graded{G}}(R))\) of right bounded objects of \(\mcalD(\Mod_{\graded{G}}(R))\) by \cite{lurie2017Higher}*{Definition 1.3.2.7 and Proposition 1.3.5.24}.
    Moreover, \(\mcalD^-(\Mod_{\graded{G}}(R))\) has a left complete \(t\)-structure which is compatible with the embedding by \cite{lurie2017Higher}*{Proposition 1.3.2.19, Definition 1.3.5.16, and Proposition 1.3.3.16}.
    Then the equivalences of abelian categories
    \begin{center}
        \begin{tikzcd}
            \Mod_{\graded{G}}(R) \arrow[d] \arrow[r, "\simeq"] & \mcalD_{\graded{G}}^{\heartsuit}(R) \arrow[d] \\
            \Mod(R) \arrow[r, "\simeq"]                           & \mcalD^{\heartsuit}(R)
        \end{tikzcd}
    \end{center}
    in \Cref{PropertiesDerivedGradedModules}\Cref{ConnectivePartDgrR} extends in an essentially unique way to
    \begin{center}
        \begin{tikzcd}
            \mcalD^-(\Mod_{\graded{G}}(R)) \arrow[d] \arrow[r, "\eta"] & \mcalD_{\graded{G}}(R) \arrow[d] \\
            \mcalD^-(\Mod(R)) \arrow[r, "\eta'"]                           & \mcalD(R)
        \end{tikzcd}
    \end{center}
    such that the horizontal functors are right \(t\)-exact and the vertical functors are forgetful functors by \cite{lurie2017Higher}*{Proposition 1.3.3.12}.
    By \cite{lurie2017Higher}*{Proposition 7.1.1.15}, the lower horizontal functor is fully faithful and the essential image is \(\mcalD^-(R)\).

    For a graded projective module \(P\), a discrete graded module \(N\), and \(n>0\), one has
    \begin{equation*}
        \pi_0\Map_{\mcalD_{\graded{G}}(R)}(P,N[n])=0.
    \end{equation*}
    Indeed, \(P\) is a retract of a coproduct of modules \(R(g)\), and the free module adjunction reduces the assertion to \(\pi_0\Map_{\mcalD(\setZ)}(\setZ,N_{-g}[n])=0\).
    Thus the projective-resolution argument of \cite{lurie2017Higher}*{Proposition 7.1.1.15} applies to graded modules and gives a commutative diagram
    \begin{center}
        \begin{tikzcd}
            \mcalD^b(\Mod_{\graded{G}}(R)) \arrow[d] \arrow[r, "\simeq"] & \mcalD_{\graded{G}}^b(R) \arrow[d] \\
            \mcalD^b(\Mod(R)) \arrow[r, "\simeq"]                           & \mcalD^b(R)
        \end{tikzcd}
    \end{center}
    of bounded objects in \(\infty\)-categories with \(t\)-structures.
    Since the \(t\)-structures on \(\mcalD^-(\Mod_{\graded{G}}(R))\) and \(\mcalD_{\graded{G}}^-(R)\) are left complete, the equivalence of bounded objects gives the equivalence of right bounded objects.
    Using the right completeness, we can conclude the desired equivalence.
\end{proof}

\begin{definition} \label{GradedPerfectComplex}
    We take the smallest stable subcategory \(\mcalD_{\graded{G}}^{\perf}(R)\) of \(\mcalD_{\graded{G}}(R)\) which contains \(R\) and is closed under retracts and grade shifting.
    If \(M \in \mcalD_{\graded{G}}(R)\) belongs to \(\mcalD_{\graded{G}}^{\perf}(R)\), we say that \(M\) is \emph{graded perfect}.
\end{definition}

\begin{proposition} \label{CompactlyGeneratedGradedCat}
    On \(\mcalD_{\graded{G}}(R)\) and \(\mcalD_{\graded{G}}^{\perf}(R)\), we have the following properties.
    \begin{enumerate}
        \item The presentable \(\infty\)-category \(\mcalD_{\graded{G}}(R)\) is (\(\omega\)-)compactly generated.
        \item An object \(M \in \mcalD_{\graded{G}}(R)\) is compact if and only if \(M\) is graded perfect.
        \item Any graded perfect object \(M \in \mcalD_{\graded{G}}(R)\) has the compact underlying \(R\)-module \(\theta(M)\), i.e., \(\theta(M)\) is compact in \(\mcalD(R)\), where \(\theta \colon \mcalD_{\graded{G}}(R) \to \mcalD(R)\) is the forgetful functor in \Cref{PropertiesDerivedGradedModules}\Cref{ForgetfulFunctorDgrR}.
        \item Along the categorical equivalence in \Cref{DerivedGradedCatEquiv}, the full subcategory \(\mcalD_{\graded{G}}^{\perf}(R)\) can be identified with the full subcategory of \(\mcalD(\Mod_{\graded{G}}(R))\) consisting of the image of bounded complex of graded finite projective \(R\)-modules.\footnote{A graded \(R\)-module is a graded finite projective \(R\)-module if it is a direct summand of a graded-free \(R\)-module of finite rank.}
    \end{enumerate}
    Especially, any object of \(\mcalD_{\graded{G}}(R)\) can be written as a (\(\omega\)-)filtered\footnote{In the sense of \cite{lurie2009Higher}*{Proposition 5.3.1.7}.} colimit of objects of \(\mcalD_{\graded{G}}^{\perf}(R)\).
\end{proposition}

\begin{proof}
    Our proof is based on \cite{lurie2017Higher}*{Proposition 7.2.4.2}.
    In the proof of (1), we will define a canonical equivalence \(\Ind(\mcalD_{\graded{G}}^{\perf}(R)) \xrightarrow{\simeq} \mcalD_{\graded{G}}(R)\) of \(\infty\)-categories. This gives the last assertion.

    (1): We first show that the shifted \(G\)-graded \(R\)-module \(R(g)\) is compact for any \(g \in G\): Take the left adjoint functor \(- \Lgrotimes_{\setZ} R \colon \mcalD_{\graded{G}}(\setZ) \to \mcalD_{\graded{G}}(R)\) given in the proof of \Cref{PropertiesDerivedGradedModules}\Cref{tStructureDgrR}. For any \(N \in \mcalD_{\graded{G}}(R)\), we have the following equivalences
    \begin{align}
        & \Map_{\mcalD_{\graded{G}}(R)}(R(g), N) \simeq \Map_{\mcalD_{\graded{G}}(R)}(R, N(-g)) \simeq \Map_{\mcalD_{\graded{G}}(\setZ)}(\setZ, N(-g)) \label{EquivCorepresentingGraded} \\
        & \simeq \Map_{\mcalD_{\graded{G}}(\setZ)}(\setZ, \tau_{\geq 0}(N(-g))) \simeq \prod_{h \in G} \Map_{\mcalD(\setZ)}(\setZ_h, \tau_{\geq 0}(N(-g))_h) \nonumber \\
        & \simeq \Map_{\mcalD(\setZ)}(\setZ, \tau_{\geq 0}(N_{-g})) \simeq \Omega^{\infty}\tau_{\geq 0}(N_{-g}) \nonumber
    \end{align}
    in \(\Ani\), where \(\Omega^{\infty}(-) \colon \mcalD(\setZ) \to \Ani\) is the functor of the underlying anima.
    Using this, the functor \(\mcalD_{\graded{G}}(R) \to \Ani\) corepresented by \(R(g)\) is equivalent to the composition of functors
    \begin{equation*}
        \mcalD_{\graded{G}}(R) \xrightarrow{\text{\Cref{PropertiesDerivedGradedModules}\Cref{LimitsColimitsDgrR}}} \mcalD_{\graded{G}}(\setZ) \xrightarrow{(-)_{-g}} \mcalD(\setZ) \xrightarrow{\tau_{\geq 0}} \mcalD_{\geq 0}(\setZ) \xrightarrow{\text{forgetful}} \Spectrum_{\geq 0} \xrightarrow{\Omega^{\infty}} \Ani
    \end{equation*}
    and all functors commute with filtered colimits by \Cref{PropertiesDgrZ}\Cref{ConnectivePartDgrZ} and \cite{lurie2017Higher}*{Proposition 1.4.3.9}.
    Therefore, \(R(g)\) is a compact object of \(\mcalD_{\graded{G}}(R)\).

    Next, we will show that the full subcategory \(\mcalD_{\graded{G}}^{\omega}(R)\) of (\(\omega\)-)compact objects of \(\mcalD_{\graded{G}}(R)\) is stable and closed under retracts and grade shifting: As in the above arguments, we can check that \(\mcalD_{\graded{G}}^{\omega}(R)\) is stable and closed under grade shifting. Moreover, we can check the closed property on retracts by using, for example, a characterization of compact objects written in \cite{lurie2017Higher}*{Proposition 1.4.4.1 (3)}.
    Consequently, the full subcategory \(\mcalD_{\graded{G}}^{\omega}(R)\) contains \(\mcalD_{\graded{G}}^{\perf}(R)\).

    The fully faithful functor \(\mcalD_{\graded{G}}^{\perf}(R) \hookrightarrow \mcalD_{\graded{G}}(R)\) induces a functor
    \begin{equation} \label{FunctorFromIndPerf}
        F \colon \Ind(\mcalD_{\graded{G}}^{\perf}(R)) \hookrightarrow \mcalD_{\graded{G}}(R),
    \end{equation}
    which is (\(\omega\)-)continuous and fully faithful by \cite{lurie2009Higher}*{Proposition 5.3.5.10 and Proposition 5.3.5.11}. Take the essential image \(\Image(F) \subseteq \mcalD_{\graded{G}}(R)\) of \(F\)

    \begin{claim} \label{IndCatCompactPerf}
        An object \(M \in \Ind(\mcalD_{\graded{G}}^{\perf}(R)) \simeq \Image(F)\) is compact if and only if \(M\) belongs to the essential image of the Yoneda embedding \(\mcalD_{\graded{G}}^{\perf}(R) \hookrightarrow \Ind(\mcalD_{\graded{G}}^{\perf}(R))\).
    \end{claim}
    \begin{claimproof}
        Any object of \(\mcalD_{\graded{G}}^{\perf}(R)\) is compact in \(\Ind(\mcalD_{\graded{G}}^{\perf}(R))\) by \cite{lurie2009Higher}*{Proposition 5.3.5.5}.
        Conversely, take any compact object \(M\) of \(\Image(F) \subseteq \mcalD_{\graded{G}}(R)\).
        Write \(M = \colim_{i \in I} M_i\) for \(\omega\)-filtered partially ordered set\footnote{See \cite{lurie2009Higher}*{Proposition 5.3.1.18}.} \(I\) and \(M_i \in \mcalD_{\graded{G}}^{\perf}(R)\).
        Since \(M\) is compact, the identity morphism on \(M\) factors through some \(M_i\) in \(\mcalD_{\graded{G}}(R)\).
        Since \(\mcalD_{\graded{G}}^{\perf}(R)\) is closed under retracts, the chosen \(M\) is also perfect.
    \end{claimproof}

    We will show that this functor \(F\) is essentially surjective.
    It will be checked that the full subcategory \(\Image(F)\) is presentable and is closed under small colimits and extensions.
    Since \(\mcalD_{\graded{G}}^{\perf}(R)\) is stable by the definition, then \(\Ind(\mcalD_{\graded{G}}^{\perf}(R)) \simeq \Image(F)\) is a stable full subcategory of \(\mcalD_{\graded{G}}(R)\) by \cite{lurie2017Higher}*{Proposition 1.1.3.6} and then it is closed under extensions.

    We will show that \(\Ind(\mcalD_{\graded{G}}^{\perf}(R))\) is presentable: Since any small coproduct can be written as a (\(\varpi\)-)filtered colimit of finite coproducts, \(\Ind(\mcalD_{\graded{G}}^{\perf}(R))\) admits small coproducts. Using \cite{lurie2017Higher}*{Proposition 1.4.4.1 (1)}, \(\Ind(\mcalD_{\graded{G}}^{\perf}(R))\) admits small colimits.
    Under the equivalence condition in \Cref{IndCatCompactPerf}, we can use \cite{lurie2009Higher}*{Corollary 5.3.4.15} to say that \(\mcalD_{\graded{G}}^{\perf}(R)\) admits (resp., is closed under) (\((\omega)\)-)small colimits in \(\Ind(\mcalD_{\graded{G}}^{\perf}(R))\).
    This implies that \(\Image(F)\) is presentable by \cite{lurie2009Higher}*{Theorem 5.5.1.1}.

    Next, it is checked that \(\Image(F) \subseteq \mcalD_{\graded{G}}(R)\) is closed under small colimits.
    Since the restriction \(\restr{F}{\Image(F)^{\omega}}\) is the same as the inclusion \(\mcalD_{\graded{G}}^{\perf}(R) \hookrightarrow \mcalD_{\graded{G}}(R)\) from the stable \(\infty\)-category \(\mcalD_{\graded{G}}^{\perf}(R)\) by \Cref{IndCatCompactPerf}, it is right exact, i.e., preserves finite colimits.
    The functor \(F\) is already (\(\omega\)-)continuous and thus \cite{lurie2009Higher}*{Proposition 5.5.1.9} shows that \(F\) preserves small colimits, which means \(\Image(F)\) is closed under small colimits.

    By \cite{lurie2017Higher}*{Proposition 1.4.4.11}, these properties give a \(t\)-structure \((\mcalD_{\geq0},\mcalD_{\leq0})\) on \(\mcalD_{\graded{G}}(R)\) with \(\mcalD_{\geq0}=\Image(F)\).
    Take any object \(M\) of \(\mcalD_{\leq-1}\). This satisfies that \(\Map_{\mcalD_{\graded{G}}(R)}(M', M) \simeq 0\) for any \(M' \in \Image(F)\).
    Especially, we can take \(R(g)[n]\) as \(M'\) for all \(g \in G\) and \(n \in \setZ\).
    The equivalences
    \begin{equation*}
        0=\pi_0\Map_{\mcalD_{\graded{G}}(R)}(R(g)[n],M)\cong\pi_0\Map_{\mcalD(\setZ)}(\setZ[n],M_{-g})\cong\pi_n(M_{-g})
    \end{equation*}
    hold as in \eqref{EquivCorepresentingGraded}.
    Therefore, \(M\) is equivalent to \(0\) in \(\mcalD_{\graded{G}}(R)\) and thus \(\Image(F) = \mcalD_{\graded{G}}(R)\).

    (2): Because of the equivalence \(\Ind(\mcalD_{\graded{G}}^{\perf}(R)) \simeq \Image(F) = \mcalD_{\graded{G}}(R)\), \Cref{IndCatCompactPerf} shows the desired result.

    (3): Recall that the full subcategory \(\mcalD^{\perf}(R)\) of perfect objects of \(\mcalD(R)\) is the smallest stable subcategory of \(\mcalD(R)\) which contains \(R\) and is closed under retracts, and it is the same as the full subcategory of compact objects of \(\mcalD(R)\) (\cite{lurie2017Higher}*{Definition 7.2.4.1 and Proposition 7.2.4.2}).
    Take the inverse image \(\theta^{-1}(\mcalD^{\perf}(R))\) in \(\mcalD_{\graded{G}}(R)\).
    This \(\theta^{-1}(\mcalD^{\perf}(R))\) contains \(R\) and is closed under retracts and grade shifting.
    Since the forgetful functor \(\theta\) is exact, we can show that \(\theta^{-1}(\mcalD^{\perf}(R))\) is stable and then it contains \(\mcalD^{\perf}_{\graded{G}}(R)\).
    This implies the desired conclusion.

    (4): Let \(\mcalC\) be the full subcategory of \(\mcalD(\Mod_{\graded{G}}(R))\) represented by bounded complexes of graded finite projective modules.
    Finite filtrations by stupid truncations show that \(\mcalC\subseteq\mcalD_{\graded{G}}^{\perf}(R)\).
    The category \(\mcalC\) is stable by the mapping-cone construction and is closed under grade shifts.
    Its homotopy category is the bounded homotopy category of the idempotent complete additive category of graded finite projective modules, and is therefore idempotent complete by \cite{schnurer2011Homotopy}*{Theorem 3.4(3)}.
    Hence \(\mcalC\) is closed under retracts in \(\mcalD(\Mod_{\graded{G}}(R))\).
    Since it contains every \(R(g)\), the definition of graded perfect objects gives the reverse inclusion.
\end{proof}

\section{Derived gradedwise complete objects} \label{SectionDerivedGradedwiseComp}

This section defines derived gradedwise completeness and the derived gradedwise completion functor and proves their basic properties (\Cref{PrincipalDerivedGrCompleteness,DerivedGradedwiseComp,graded-derived-nakayama,DerivedGrcompComplex}).

\subsection{Derived gradedwise completeness}

We first define derived gradedwise completeness, similarly to derived completeness for non-graded objects (\Cref{PrincipalDerivedGrCompleteness}).

\begin{notation}\label{NotationDerivedQuot}
    In this subsection, we fix a \(G\)-graded ring \(R\) and a finitely generated homogeneous ideal \(I\) of \(R\).
    Let \(M\) be an object of \(\mcalD_{\graded{G}}(R)\).
    \begin{itemize}
        \item For a sequence of homogeneous elements \(f_1, \dots, f_r\) of \(R\), we define the \emph{derived quotient of \(M\) by \(f_1, \dots, f_r\)} by
    \begin{equation*}
        M/^L (f_1, \dots, f_r) \defeq M \Lgrotimes_{R[X_1, \dots, X_r]} R[X_1, \dots, X_r]/(X_1, \dots, X_r) \in \mcalD_{\graded{G}}(R),
    \end{equation*}
    where \(R[X_1, \dots, X_r]\) is the polynomial ring with \(r\) variables \(X_i\) which has the \(G\)-graded structure by \(\deg(X_i) = \deg(f_i)\) and the \(G\)-graded \(R[X_1, \dots, X_r]\)-module structure on \(M\) is given by $X_i m \defeq f_i m$ for $m \in M$.
    In particular, we have
    \begin{align}\label{derived-quotient}
            M \Lgrotimes_R R/^L(f_1^n, \dots, f_r^n)
            &\cong M \Lgrotimes_R (R \Lgrotimes_{R[\underline{X}]} R[\underline{X}]/(\underline{X})) \\
            &\cong M \Lgrotimes_{R[\underline{X}]} R[\underline{X}]/(\underline{X}) = M/^L(f_1^n, \dots, f_r^n), \nonumber
    \end{align}
    in \(\mcalD_{\graded{G}}(R)\), where $\underline{X}\defeq(X_1,\ldots,X_r)$, \(\deg(X_i)=n\deg(f_i)\), and the left \(R\)-factor and \(M\) are regarded as \(R[\underline{X}]\)-modules via \(X_i\mapsto f_i^n\).

    If \(r = 1\), taking the base change \(M \Lgrotimes_{R[X]} -\) of the fiber sequence
    \begin{equation*}
        R[X](-\deg(X)) \xrightarrow{\times X} R[X] \to R[X]/(X)
    \end{equation*}
    in \(\mcalD_{\graded{G}}(R[X])\) gives the fiber sequence
    \begin{equation} \label{DerivedQuotCofiber}
        M(-\deg(f)) \xrightarrow{\times f} M \to M/^L f
    \end{equation}
    in \(\mcalD_{\graded{G}}(R)\).
    \item For any homogeneous element \(f \in R\), we denote by \(T(M, f)\) the limit
    \begin{equation*}
        T(M, f) \defeq \grlim(\cdots \xrightarrow{f} M(-2\deg(f)) \xrightarrow{f} M(-\deg(f)) \xrightarrow{f} M)
    \end{equation*}
    in \(\mcalD_{\graded{G}}(R)\). This is functorial on \(M\).

    Taking the limit of the cofiber sequences \(M(-n\deg(f))\xrightarrow{f^n}M\to M/^L f^n\), we have a fiber sequence
    \begin{equation*}
        T(M, f) \to M \to \grlim_{n \geq 0} M/^L f^n
    \end{equation*}
    in \(\mcalD_{\graded{G}}(R)\).
\end{itemize}
\end{notation}

\begin{remark} \label{AnotherTMf}
    Let \(f\in R_d\) be homogeneous.
    The graded localization has the presentation
    \begin{equation*}
        R[1/f]\simeq\grcolim\bigl(R\xrightarrow{f}R(d)\xrightarrow{f}R(2d)\to\cdots\bigr).
    \end{equation*}
    Applying the graded internal Hom of \Cref{PropertiesDerivedGradedModules}\Cref{ClosedMonoidalDgrR} gives
    \begin{equation*}
        \grmapspt_R(R[1/f],M)\simeq\grlim\bigl(\cdots\xrightarrow{f}M(-2d)\xrightarrow{f}M(-d)\xrightarrow{f}M\bigr)=T(M,f).
    \end{equation*}
\end{remark}

\begin{definition} \label{PrincipalDerivedGrCompleteness}
    In the setting of \cref{NotationDerivedQuot}, we say that \(M\) is \emph{derived gradedwise \(I\)-complete} if for any homogeneous element \(f \in I\), the limit \(T(M, f)\) is zero in \(\mcalD_{\graded{G}}(R)\).

    We denote by \(\mcalD_{\graded{G}}^{\comp{I}}(R) \subseteq \mcalD_{\graded{G}}(R)\) the full subcategory spanned by derived gradedwise \(I\)-complete objects.

    For any homogeneous element \(g \in R\), the \emph{derived gradedwise \(g\)-completion \(\dgrcomp{g}{M}\) of \(M\)} is defined by
    \begin{equation*}
        \dgrcomp{g}{M}  \defeq \grlim_{n \geq 0} M/^L g^n.
    \end{equation*}
    We note that we have the canonical morphism \(M \to \dgrcomp{g}{M}\) in \(\mcalD_{\graded{G}}(R)\) and the fiber sequence
    \begin{equation*}
        T(M, g) \to M \to  \dgrcomp{g}{M}.
    \end{equation*}
\end{definition}

\subsection{Derived gradedwise completion}

We define the derived gradedwise completion functor for finitely many homogeneous elements and then for finitely generated homogeneous ideals (\Cref{DerivedGradedwiseComp-gene,DerivedGradedwiseComp}).

\begin{definition} \label{DerivedGradedwiseComp-gene}
    In the setting of \cref{NotationDerivedQuot}, for homogeneous elements \(f_1, \ldots, f_r\), we define
    \begin{equation*}
        \dgrcomp{f_1, \dots, f_r}{M} \defeq
        \dgrcomp{f_r} \circ \dgrcomp{f_{r-1}} \circ \cdots \circ \dgrcomp{f_1}{M}
        \in \mcalD_{\graded{G}}(R).
    \end{equation*}
\end{definition}

\begin{lemma}\label{ind-generator}
In the setting of \cref{NotationDerivedQuot}, let \(f_1,\ldots,f_r\) be homogeneous elements and set \(I=(f_1,\ldots,f_r)\).
Then:
\begin{enumerate}
    \item \(\dgrcomp{f_1,\ldots,f_r}{M}\) is derived gradedwise \(I\)-complete.
    \item The canonical morphism \(M \to \dgrcomp{f_1,\ldots,f_r}{M}\) is the universal morphism from \(M\) to a derived gradedwise \(I\)-complete object, that is, for any derived gradedwise \(I\)-complete object \(L\) in \(\mcalD_{\graded{G}}(R)\), the natural map
    \begin{equation*}
        \Map_{\mcalD_{\graded{G}}(R)}(\dgrcomp{f_1,\ldots,f_r}{M}, L) \xrightarrow{\ \simeq\ } \Map_{\mcalD_{\graded{G}}(R)}(M, L)
    \end{equation*}
    is an equivalence, where the map is induced by composition with \(M \to \dgrcomp{f_1,\ldots,f_r}{M}\).
\end{enumerate}
In particular, the inclusion functor \(\mcalD_{\graded{G}}^{\comp{I}}(R) \hookrightarrow \mcalD_{\graded{G}}(R)\) has a left adjoint \(\dgrcomp{f_1,\ldots,f_r}{-}\).
\end{lemma}

\begin{proof}
    For each homogeneous \(f\), the functor \(-/^L f^n\) is a finite cofiber construction involving grade shifts, and therefore preserves limits.
    Commuting these constructions with the successive limits defining completion, and restricting to the diagonal in the indexing inverse systems, gives
    \begin{equation*}
        \dgrcomp{f_1,\ldots,f_r}{M}\simeq\grlim_{n\geq1}M/^L(f_1^n,\ldots,f_r^n).
    \end{equation*}
    The object \(M/^L(f_1^n,\ldots,f_r^n)\) is a module over the derived quotient algebra \(R/^L(f_1^n,\ldots,f_r^n)\), whose zeroth homotopy ring is \(R/(f_1^n,\ldots,f_r^n)\).
    Thus its homotopy groups are annihilated by \((f_1^n,\ldots,f_r^n)\).
    For every homogeneous \(h\in I\), the containment \(I^{r(n-1)+1}\subseteq(f_1^n,\ldots,f_r^n)\) makes the inverse systems defining \(T(M/^L(f_1^n,\ldots,f_r^n),h)\) uniformly nilpotent on homotopy groups.
    Both inverse limits and their first derived limits vanish, so \Cref{MilnorExactSequence} gives
    \begin{equation*}
        T(M/^L(f_1^n,\ldots,f_r^n),h)=0.
    \end{equation*}
    Since \(T(-,h)\) preserves limits, this proves (1).

    For (2), first take \(r=1\) and write \(f=f_1\), \(d=\deg(f)\).
    Removing the first term of the tower defining \(T(M,f)\) shows that multiplication \(f:T(M,f)(-d)\to T(M,f)\) is an isomorphism.
    Hence \(T(M,f)\Lgrotimes_R R[1/f]\simeq T(M,f)\).
    If \(L\) is derived gradedwise \((f)\)-complete, \Cref{AnotherTMf} and tensor--internal-Hom adjunction give
    \begin{equation*}
        \Map_{\mcalD_{\graded{G}}(R)}(T(M,f),L)\simeq\Map_{\mcalD_{\graded{G}}(R)}\bigl(T(M,f),\grmapspt_R(R[1/f],L)\bigr)\simeq *.
    \end{equation*}
    Applying \(\Map_{\mcalD_{\graded{G}}(R)}(-,L)\) to the fiber sequence \(T(M,f)\to M\to\dgrcomp{f}{M}\) proves the required equivalence.
    For arbitrary \(r\), a derived gradedwise \(I\)-complete object is derived gradedwise \((f_i)\)-complete for every \(i\), so successive applications of the principal case prove (2).
\end{proof}

\begin{definition} \label{DerivedGradedwiseComp}
In the setting of \cref{NotationDerivedQuot}, for a homogeneous ideal \(I \subset R\), the \emph{derived gradedwise \(I\)-completion of \(M\)} is defined to be the value at \(M\) of a left adjoint to the inclusion functor
\[
\mcalD_{\graded{G}}^{\comp{I}}(R) \hookrightarrow \mcalD_{\graded{G}}(R).
\]
We denote such a functor by \(\dgrcomp{I}{M}\).
By \cref{ind-generator}, for any choice of homogeneous generators \(f_1,\ldots,f_r\) of \(I\),
the object \(\dgrcomp{f_1,\ldots,f_r}{M}\) is a derived gradedwise \(I\)-completion of \(M\).

\end{definition}

\begin{lemma} \label{EquivDerivedGradedwiseCompletenessHomology}
    In the setting of \Cref{NotationDerivedQuot}, the following are equivalent.
    \begin{enumerate}
        \item The given \(M \in \mcalD_{\graded{G}}(R)\) is derived gradedwise \(I\)-complete.
        \item Any homotopy group \(\pi_i(M) \in \Mod_{\graded{G}}(R)\) of \(M\) is derived gradedwise \(I\)-complete for any \(i \in \setZ\) as an object of \(\mcalD_{\graded{G}}(R)\).
    \end{enumerate}
\end{lemma}

\begin{proof}
    We follow the proof of \cite{lurie2018Spectral}*{Theorem 7.3.4.1}.
    Take any homogeneous element \(f \in I\).
    Set \(\setZ_{\geq 0}\)-indexed diagrams
    \begin{align*}
        D(M) & \defeq (\cdots \xrightarrow{\times f} M(-2\deg(f)) \xrightarrow{\times f} M(-\deg(f)) \xrightarrow{\times f} M) \quad \text{and} \\
        D_i(M) & \defeq (\cdots \xrightarrow{\times f} \pi_i(M)(-2\deg(f)) \xrightarrow{\times f} \pi_i(M)(-\deg(f)) \xrightarrow{\times f} \pi_i(M))
    \end{align*}
    in \(\mcalD_{\graded{G}}(R)\) and \(\Mod_{G-gr}(R)\) whose limit (in \(\mcalD_{\graded{G}}(R)\)) is \(T(M, f)\) and \(T(\pi_i(M), f)\) respectively.
    Applying the Milnor exact sequence (\Cref{MilnorExactSequence}) for \(D\), we have a short exact sequence
    \begin{equation} \label{MilnorExactCompleteness}
        0 \to R^1\!\grlim_{n \geq 0}D_{i+1}(M) \to \pi_i(T(M, f)) \to R^0\!\grlim_{n \geq 0}D_i(M) \to 0
    \end{equation}
    of discrete \(G\)-graded \(R\)-modules.
    Because of \(R^i\grlim(-) = \pi_{-i}(\grlim(-))\), this exact sequence is
    \begin{equation*}
        0 \to \pi_{-1}(T(\pi_{i+1}(M), f)) \to \pi_i(T(M, f)) \to \pi_0(T(\pi_i(M), f)) \to 0.
    \end{equation*}
    Therefore, \(M\) is derived gradedwise \(I\)-complete if and only if \(\pi_k(T(\pi_i(M), f))\) is zero for \(k = 0, -1\) for all \(i \in \setZ\).
    So it suffices to show that \(T(M, f)\) is concentrated in (homological) degree \([-1, 0]\) if \(M\) is discrete.
    This also follows from the above Milnor exact sequence: Indeed, if \(M\) satisfies \(\pi_i(M) = 0\) for \(i \neq 0\), the diagram \(D_i(M)\) is zero for \(i \neq 0\) and thus the exact sequence \eqref{MilnorExactCompleteness} shows the vanishing of \(\pi_i(T(M, f))\) for \(i \neq 0, -1\).
\end{proof}

\begin{lemma} \label{PrincipalDerivedGradedwiseCompProp}
    In the setting of \cref{NotationDerivedQuot}, the following properties are satisfied:
    \begin{enumerate}
        \item \label{DerivedGrcompLimits} The full subcategory \(\mcalD_{\graded{G}}^{\comp{I}}(R)\) is stable under limits. In particular, this full subcategory is a stable \(\infty\)-category.
        \item \label{DerivedGrcompPreservesLimits} Taking the derived gradedwise \(I\)-completion functor \(\dgrcomp{I}{-} \colon \mcalD_{\graded{G}}(R) \to \mcalD_{\graded{G}}^{\comp{I}}(R)\) preserves limits.
        \item \label{DerivedGrcompComplete} The derived gradedwise \(I\)-completion \(\dgrcomp{I}{M}\) is derived gradedwise \(I\)-complete.
        \item \label{DerivedGrcompUniv} The canonical morphism \(M \to \dgrcomp{I}{M}\) is the universal morphism from \(M\) to a derived gradedwise \(I\)-complete object. In particular, the inclusion functor \(\mcalD_{\graded{G}}^{\comp{I}}(R) \hookrightarrow \mcalD_{\graded{G}}(R)\) has a left adjoint \(\dgrcomp{I}{-}\).
        \item \label{DerivedGrcompQuotLimit} The derived gradedwise \(I\)-completion \(\dgrcomp{I}{M}\) can be written as the limit
        \begin{equation*}
            \dgrcomp{I}{M} \cong \grlim_{n \geq 0} (M/^L (f_1^{n}, \ldots, f_r^{n}))
        \end{equation*}
        where \(M/^L (f_1^n, \dots, f_r^n)\) is the derived quotient defined in \Cref{NotationDerivedQuot}.
        \item \label{DerivedGrcompQuot} If $M$ is derived gradedwise $I$-complete, then so is $M/^L(g_1,\ldots,g_l)$ for any homogeneous elements $g_1,\ldots,g_l \in R$.
        \item \label{DerivedGrcompPCompletion} In the case of \(I = (p)\) for a prime number \(p\), the derived gradedwise \(p\)-completion \(\dgrcomp{p}{M}\) is the same as the composition \(G^{ds} \xrightarrow{M} \mcalD(\setZ) \xrightarrow{\dcomp{p}{-}} \mcalD(\setZ)\), where \(\dcomp{p}{-}\) is the usual derived \(p\)-completion on \(\mcalD(\setZ)\).
        \item \label{BoundedTorsionDerivedCompletion} In the case of \(I = (f)\), if \(M\) is concentrated in degree \(0\) and has bounded \(f^{\infty}\)-torsion, then the canonical morphism
        \begin{equation*}
            \dgrcomp{f}{M} \to \grcomp{f}{M}
        \end{equation*}
        is an isomorphism in \(\mcalD_{\graded{G}}(R)\).
        \item \label{DerivedGrcompAdicCompletion} The derived \(I\)-completion \(\dcomp{I}{M}\) of the underlying \(R\)-module of \(M\) and the one of \(\dgrcomp{I}{M}\) are the same.
        \item \label{DerivedGrcompGrcomp} Assume that \(M\) is concentrated in degree \(0\) and is \(I\)-adically gradedwise complete, i.e., the canonical morphism
\[
M \xrightarrow{\cong} \bigoplus_{g\in G}\lim_{n\geq1}\bigl(M_g/(I^nM)_g\bigr)
\]
is an isomorphism of discrete \(G\)-graded \(R\)-modules (\cite{ishizuka2025Graded}*{Construction 3.3}).
        Then \(M\) is derived gradedwise \(I\)-complete.
        \item \label{DerivedGrcompAbelian} The full subcategory \(\Mod_{\graded{G}}^{\comp{I}}(R)\) of \(\Mod_{\graded{G}}(R)\) spanned by derived gradedwise \(I\)-complete discrete \(G\)-graded \(R\)-modules is an abelian category.
    \end{enumerate}
\end{lemma}

\begin{proof}
    \Cref{DerivedGrcompLimits}: Each \(T(-,f)\) preserves limits and commutes with suspension, so its kernel is stable and closed under limits.
    Intersecting these kernels over the homogeneous elements of \(I\) proves the assertion.

    \Cref{DerivedGrcompPreservesLimits}: We may assume that \(r = 1\) and \(I = (g)\). This follows from the definition of \(\dgrcomp{g}{M}\) as the cofiber of \(T(M, g) \to M\) in \(\mcalD_{\graded{G}}(R)\).

    \Cref{DerivedGrcompComplete} and \Cref{DerivedGrcompUniv} follow from \cref{ind-generator}.

    \Cref{DerivedGrcompQuotLimit}: We will only prove the case that \(r = 2\) because the general case is similar. Since taking the derived gradedwise \(f_i\)-completion preserves limits by \Cref{DerivedGrcompPreservesLimits}, we have
    \begin{align*}
        \dgrcomp{(f_1, f_2)}{M} & \cong \grlim_{n \geq 0} \dgrcomp{f_2}{M/^L f_1^n} \cong \grlim_{n \geq 0} \grlim_{m \geq 0} (M/^L f_1^n)/^L f_2^m \cong \grlim_{n \geq 0} M/^L (f_1^n, f_2^n)
    \end{align*}
    where the second isomorphism follows from that \(M/^L f_1^n\) is the cofiber of \(M(-n\deg(f_1))\xrightarrow{f_1^n}M\) and this commutes with limits.

    \Cref{DerivedGrcompQuot}: By induction on $l$, we may assume $l=1$ and set $g \defeq g_1$.
    For $f \in R$, we have
    \begin{equation*}
        \dgrcomp{f}{M/^L(g)} \simeq \grlim_{n \geq 0} M/^L(g,f^n)
        \simeq \parenlr{\grlim_{n \geq 0} M/^L(f^n)}/^L g
        \simeq \dgrcomp{f}{M}/^L g.
    \end{equation*}
    Hence, if $M$ is derived gradedwise $I$-complete, then so is $M/^L(g)$, as desired.

    \Cref{DerivedGrcompPCompletion}: First note that \(\dgrcomp{p}{M} = \grlim_{n \geq 0} M/^L p^n\) can be written as the limit \(\grlim_{n \geq 0} (M \Lgrotimes_{\setZ} \setZ/p^n\setZ) \in \mcalD_{\graded{G}}(R)\) which inherits the graded \(R\)-module structure by the one on \(M\).
    Since all limits in \(\mcalD_{\graded{G}}(R)\) are computed gradedwise, the limit gives the derived \(p\)-completion of each graded \(M_g \in \mcalD(\setZ)\) of \(M\) and this is the composition \(G^{ds} \xrightarrow{M} \mcalD(\setZ) \xrightarrow{\dcomp{p}{-}} \mcalD(\setZ)\).

    \Cref{BoundedTorsionDerivedCompletion}: If \(M\) has bounded \(f^\infty\)-torsion, then the pro-systems \(\{M/^L f^n\}_{n \geq 1}\) and \(\{M/f^nM\}_{n \geq 1}\) are pro-isomorphic along the canonical morphism \(M/^L f^n \to M/f^nM\) by \Cref{BoundedProIsom}.
    Therefore, taking the limits in \(\mcalD_{\graded{G}}(R)\), we have an isomorphism \(\dgrcomp{f}{M} \xrightarrow{\cong} \grlim_{n \geq 0}M/f^nM\) in \(\mcalD_{\graded{G}}(R)\).
    Then \Cref{MittagLefflerDiscrete} shows that the target is the same as the \(f\)-adic gradedwise completion of \(M\) defined in \cite{ishizuka2025Graded}.
    Therefore, we have the desired isomorphism.

    \Cref{DerivedGrcompAdicCompletion}: Let \(\theta:\mcalD_{\graded{G}}(R)\to\mcalD(R)\) be the forgetful functor.
    The ordinary derived completion of \(\theta(\dgrcomp{I}{M})\) is
    \begin{equation*}
        \dcomp{I}{\theta(\dgrcomp{I}{M})}\simeq\lim_{n\geq1}\theta\bigl((\dgrcomp{I}{M})/^L(f_1^n,\ldots,f_r^n)\bigr).
    \end{equation*}
    By \eqref{derived-quotient} and the corresponding formula for \(\dcomp{I}{\theta(M)}\), it suffices to show that the canonical morphism
    \begin{equation*}
        M/^L(f_1^n, \dots, f_r^n) \to \dgrcomp{I}{M}/^L(f_1^n, \dots, f_r^n)
    \end{equation*}
    is an isomorphism in \(\mcalD_{\graded{G}}(R)\) for all \(n \geq 1\).
    Using the fiber sequence \eqref{DerivedQuotCofiber}, taking the derived quotient commutes with limits.
    Therefore, the above morphism is the same as
    \begin{equation*}
        M/^L(f_1^n, \dots, f_r^n) \to \dgrcomp{I}{M/^L(f_1^n, \dots, f_r^n)}
    \end{equation*}
    for all \(n \geq 1\).
    Since \(f_1, \dots, f_r\) are homogeneous generators of \(I\) in \(R\), the derived quotient \(M/^L(f_1^n, \dots, f_r^n)\) is already derived gradedwise \(I\)-complete as in the proof of \Cref{ind-generator}(1).
    Then the above morphism must be an isomorphism by \Cref{PrincipalDerivedGradedwiseCompProp}\Cref{DerivedGrcompUniv}.
    This finishes the proof of \Cref{DerivedGrcompAdicCompletion}.

    \Cref{DerivedGrcompGrcomp}: By \Cref{MittagLefflerDiscrete}, the assumed classical completeness identifies \(M\) with \(\grlim_{n\geq1}M/I^nM\) in \(\mcalD_{\graded{G}}(R)\).
    For homogeneous \(f\in I\), multiplication by \(f^n\) vanishes on \(M/I^nM\), so \Cref{MilnorExactSequence} gives \(T(M/I^nM,f)=0\).
    Passing to the limit proves \(T(M,f)=0\).

    \Cref{DerivedGrcompAbelian}: For a morphism \(u:M\to N\) between derived gradedwise \(I\)-complete discrete modules, \(\cofib(u)\) is derived gradedwise \(I\)-complete.
    By \Cref{EquivDerivedGradedwiseCompletenessHomology}, its homotopy groups \(\pi_1(\cofib(u))=\Ker(u)\) and \(\pi_0(\cofib(u))=\operatorname{coker}(u)\) are also derived gradedwise \(I\)-complete.
    Thus the full additive subcategory is closed under kernels and cokernels in \(\Mod_{\graded{G}}(R)\), and is abelian.
\end{proof}

\subsection{Nakayama type results and other properties}

As in the derived Nakayama lemma (e.g., \citeSta{0G1U} and \citeSta{0H83}), we have the following graded version for derived gradedwise complete objects (\Cref{graded-derived-nakayama}).

\begin{proposition}\label{graded-derived-nakayama}
In the setting of \Cref{NotationDerivedQuot}, assume that \(M\) is derived gradedwise \(I\)-complete.
Let $f_1,\ldots,f_r$ be homogeneous generators of $I$, and let $m \in \Z$.
\begin{enumerate}
    \item If $\pi_i(M/^L(f_1,\ldots,f_r))=0$ for $i<m$, then $\pi_i(M)=0$ for $i<m$.
    \item If \(\pi_i(M \Lgrotimes_R R/I) = 0\) for \(i < m\), then \(\pi_i(M) = 0\) for \(i < m\).
\end{enumerate}
\end{proposition}

\begin{proof}
By replacing $M$ with $M[-m]$, we may assume $m=0$.

We prove (1) by induction on $r$.
For $r=1$, set $f \defeq f_1$.
Since $M$ is derived gradedwise $(f)$-complete, we have
\begin{equation*}
    M \cong \grlim_{n \geq 0} M/^L f^n,
\end{equation*}
by \Cref{PrincipalDerivedGradedwiseCompProp}\Cref{DerivedGrcompUniv} and \Cref{DerivedGrcompQuotLimit} and the Milnor exact sequence (\Cref{MilnorExactSequence}) gives a short exact sequence
\begin{equation*}
    0 \to R^1\!\grlim_{n \geq 0} \pi_{i+1}(M/^L f^n) \to \pi_i(M) \to R^0\!\grlim_{n \geq 0} \pi_i(M/^L f^n) \to 0
\end{equation*}
of discrete \(G\)-graded \(R\)-modules.
The cofiber sequences
\begin{equation*}
    (M/^L f^{n-1})(-\deg(f)) \xrightarrow{\times f} M/^L f^n\to M/^L f
\end{equation*}
associated with the factorization \(f^n=f\circ f^{n-1}\) show inductively that every \(M/^L f^n\) is connective.
It follows that \(M\) is \((-1)\)-connective, i.e. \(\pi_i(M)=0\) for \(i<-1\).
Thus it remains to show that \(\pi_{-1}(M)=0\).

The fiber sequence
\[
M(-\deg(f)) \xrightarrow{\times f} M \to M/^L f
\]
together with the connectivity of $M/^L f$ implies that multiplication by $f$ induces a surjective map
\[
\times f \colon \pi_{-1}(M)(-\deg(f)) \to \pi_{-1}(M).
\]
By \Cref{MittagLefflerDiscrete}, the object \(T(\pi_{-1}(M), f)\) is concentrated in degree $0$, and the canonical morphism
\[
T(\pi_{-1}(M), f) \to \pi_{-1}(M)
\]
is surjective.
On the other hand, by \Cref{EquivDerivedGradedwiseCompletenessHomology}, the group \(\pi_{-1}(M)\) is derived gradedwise $(f)$-complete, hence \(T(\pi_{-1}(M), f)=0\).
Consequently, \(\pi_{-1}(M)=0\), as required.

Now assume $r \geq 2$.
Since \(M/^L(f_1,\ldots,f_{r-1})\) is derived gradedwise $(f_r)$-complete by \cref{PrincipalDerivedGradedwiseCompProp}\Cref{DerivedGrcompQuot} and
\[
(M/^L(f_1,\ldots,f_{r-1}))/^L(f_r) \simeq M/^L(f_1,\ldots,f_r)
\]
 is connective by assumption, it follows that $M/^L(f_1,\ldots,f_{r-1})$ is connective by the case $r=1$.
Applying the induction hypothesis again, we conclude that $M$ is connective, proving (1).

We prove (2).
Assume \(\pi_i(M \Lgrotimes_R R/I)=0\) for \(i<0\).
By the symmetric monoidal property and the $t$-exactness established in \Cref{PropertiesDerivedGradedModules}\Cref{ForgetfulFunctorDgrR}, the object \(M \otimes^L_R R/I\) is connective in \(\mcalD(R)\).
Moreover, by \citeSta{0H82}, the derived tensor product
\begin{equation*}
    M \otimes^L_R R/^L(f_1,\dots,f_r)
\end{equation*}
is connective in \(\mcalD(R)\), since \(R/^L(f_1,\dots,f_r)\) is connective, bounded, and cohomologically \(I^\infty\)-torsion by \cref{to-nakayama}.
Using the isomorphism \eqref{derived-quotient} and the $t$-exactness again, we deduce that \(M/^L(f_1,\dots,f_r)\) is connective in \(\mcalD_{\graded{G}}(R)\) (not only in \(\mcalD(R)\)).
Therefore, the connectivity of $M$ follows from part (1).
\end{proof}

\begin{corollary} \label{NakayamaIsom}
    In the setting of \Cref{NotationDerivedQuot}, take a morphism \(\varphi \colon M \to N\) of \(\mcalD_{\graded{G}}^{\comp{I}}(R)\) such that the derived \(I\)-completion
    \begin{equation*}
        \dcomp{I}{\varphi} \colon \dcomp{I}{M} \to \dcomp{I}{N}
    \end{equation*}
    is an isomorphism in \(\mcalD(R)\).
    Then \(\varphi\) itself is an isomorphism in \(\mcalD_{\graded{G}}(R)\).
    In other words, the functor
    \begin{equation*}
        \mcalD_{\graded{G}}^{\comp{I}}(R) \to \mcalD^{\comp{I}}(R); \quad M \mapsto \dcomp{I}{M}
    \end{equation*}
    is conservative, which will be written as \(\mcalF^I\) in the sequel (\Cref{ConstComonad}).
\end{corollary}

\begin{proof}
    Take the fiber \(F\) of \(\varphi\) in \(\mcalD_{\graded{G}}(R)\).
    The derived gradedwise \(I\)-completeness of \(M\) and \(N\) ensure that so is \(F\).
    By \Cref{graded-derived-nakayama}, it suffices to show that \(F/^L(f_1, \dots, f_r)\) vanishes in \(\mcalD_{\graded{G}}(R)\).
    In \(\mcalD(R)\), we have a fiber sequence
    \begin{equation*}
        \dcomp{I}{F}\to\dcomp{I}{M}\xrightarrow{\dcomp{I}{\varphi}}\dcomp{I}{N}.
    \end{equation*}
    Since \(\dcomp{I}{\varphi}\) is an isomorphism, its fiber \(\dcomp{I}{F}\) vanishes in \(\mcalD(R)\).
    Taking the derived quotient,
    \begin{equation*}
        0 = \dcomp{I}{F}/^L(f_1, \dots, f_r) \cong F/^L(f_1, \dots, f_r)
    \end{equation*}
    holds in \(\mcalD(R)\).
    Since the right hand side is the image of \(F/^L(f_1, \dots, f_r)\) in \(\mcalD_{\graded{G}}(R)\), the conservativity in \Cref{PropertiesDerivedGradedModules}\Cref{ForgetfulFunctorDgrR} shows the vanishing of \(F/^L(f_1, \dots, f_r)\) in \(\mcalD_{\graded{G}}(R)\).
\end{proof}

\begin{proposition} \label{SymmetricMonoidalCompletion}
    In the setting of \Cref{NotationDerivedQuot}, there exists an essentially unique symmetric monoidal structure \(- \cLgrotimes -\) on \(\mcalD_{\graded{G}}^{\comp{I}}(R)\) such that the derived gradedwise \(I\)-completion functor \(\dgrcomp{I}{-} \colon \mcalD_{\graded{G}}(R) \to \mcalD_{\graded{G}}^{\comp{I}}(R)\) is symmetric monoidal.
\end{proposition}

\begin{proof}
    Write \(I=(f_1,\ldots,f_r)\) with homogeneous generators.
    As proved in \Cref{PrincipalDerivedGradedwiseCompProp}\Cref{DerivedGrcompAdicCompletion}, completion induces an isomorphism
    \begin{equation*}
        M/^L(f_1^n,\ldots,f_r^n)\xrightarrow{\simeq}(\dgrcomp{I}{M})/^L(f_1^n,\ldots,f_r^n)
    \end{equation*}
    for every \(n\geq1\).
    If \(u:M\to M'\) becomes an isomorphism after derived gradedwise completion, the preceding comparison shows that \(u/^L(f_1^n,\ldots,f_r^n)\) is an isomorphism for every \(n\).
    For any \(N\), associativity of the relative tensor product gives
    \begin{equation*}
        (M\Lgrotimes_R N)/^L(f_1^n,\ldots,f_r^n)\simeq(M/^L(f_1^n,\ldots,f_r^n))\Lgrotimes_R N.
    \end{equation*}
    The quotient-tower formula therefore shows that \(u\Lgrotimes_R N\) also becomes an isomorphism after completion.
    Thus the reflective localization \(\dgrcomp{I}{-}\) is compatible with tensor products.
    By \cite{lurie2017Higher}*{Proposition 2.2.1.9}, it induces the asserted symmetric monoidal structure, with tensor product \(M\cLgrotimes_R N=\dgrcomp{I}{M\Lgrotimes_R N}\) and unit \(\dgrcomp{I}{R}\).
\end{proof}

\begin{lemma} \label{DerivedGrcompComplex}
    Take a \(G\)-graded ring \(R\) and a finitely generated homogeneous ideal \(I\) of \(R\).
    Let \(M^{\bullet}\) be a bounded complex \(M^{\bullet} = (\cdots \to M^{n-1} \to M^n \to \cdots)\) of discrete \(G\)-graded \(R\)-modules \(M^n\), i.e., an object of \(\Ch^b(\Mod_{\graded{G}}(R))\).
    Assume that the canonical morphism
    \begin{equation*}
        \dgrcomp{I}{M^n} \xrightarrow{\cong} \grcomp{I}{M^n}
    \end{equation*}
    is an isomorphism in \(\mcalD_{\graded{G}}(R)\).
    Then on the image \(M\) of \(M^{\bullet}\) to \(\mcalD_{\graded{G}}(R)\) along the equivalence \(\mcalD(\Mod_{\graded{G}}(R)) \xrightarrow{\simeq} \mcalD_{\graded{G}}(R)\) in \Cref{DerivedGradedCatEquiv}, the canonical morphism
    \begin{equation*}
        \dgrcomp{I}{M} \xrightarrow{\cong} \grcomp{I}{M^{\bullet}} \defeq (\cdots \to \grcomp{I}{M^{n-1}} \to \grcomp{I}{M^n} \to \cdots)
    \end{equation*}
    is an isomorphism in \(\mcalD_{\graded{G}}(R)\).
\end{lemma}

\begin{proof}
    First, we notice that the canonical morphism \(\dgrcomp{I}{M} \to \grcomp{I}{M^{\bullet}}\) exists:
    Indeed, since the category of derived gradedwise \(I\)-complete modules is abelian by \Cref{PrincipalDerivedGradedwiseCompProp}\Cref{DerivedGrcompAbelian} and all \(I\)-adically gradedwise complete discrete \(G\)-graded modules are derived gradedwise \(I\)-complete by \Cref{PrincipalDerivedGradedwiseCompProp}\Cref{DerivedGrcompGrcomp}, any homology groups of \(\grcomp{I}{M^{\bullet}}\) are derived gradedwise \(I\)-complete and then the canonical morphism \(M^{\bullet} \to \grcomp{I}{M^{\bullet}}\) uniquely factors through \(\dgrcomp{I}{M}\) in \(\mcalD_{\graded{G}}(R)\) by \Cref{ind-generator}.

    Since \(M^{\bullet}\) is bounded, we may assume that \(M^{\bullet}\) is concentrated in cohomological degree \([0, a]\).
    We will proceed with the proof by induction on the length of the complex \(M^{\bullet}\).
    If \(a = 0\), the complex \(M^{\bullet} = M\) is the same as a discrete \(G\)-graded \(R\)-module and the result is just the assumption.

    We take the stupid truncation of \(M^{\bullet}\), i.e., we have a short exact sequence
    \begin{equation*}
        0 \to \sigma^{\geq 1}M^{\bullet} \to M^{\bullet} \to M^0[0] \to 0
    \end{equation*}
    of chain complexes of discrete \(G\)-graded \(R\)-modules, where the morphism \(\sigma^{\geq 1}M^{\bullet} \to M^{\bullet}\) is given by
    \begin{center}
        \begin{tikzcd}
            \sigma^{\geq 1}M^{\bullet} \defeq (0 \arrow[d] \arrow[r] & M^1 \arrow[r] \arrow[d, "\id"] & M^2 \arrow[r] \arrow[d, "\id"] & \cdots \arrow[r] & M^a \arrow[d, "\id"] \arrow[r] & 0) \arrow[d, "\id"] \\
            M^{\bullet} = (M^0 \arrow[r]                             & M^1 \arrow[r]                  & M^2 \arrow[r]                  & \cdots \arrow[r] & M^a \arrow[r]                  & 0).
        \end{tikzcd}
    \end{center}
    We can take a sequence
    \begin{equation*}
        \sigma^{\geq 1}M^{\bullet} \to M \to M^0[0]
    \end{equation*}
    in \(\mcalD_{\graded{G}}(R)\).
    This is a fiber sequence: By \Cref{PropertiesDerivedGradedModules}\Cref{LimitsColimitsDgrR}, we may assume that \(R = \setZ\). Using \Cref{PropertiesDgrZ}\Cref{LimitsColimitsDgrZ}, it suffices to show that the sequence \(\sigma^{\geq 1}M_g^{\bullet} \to M_g \to M^0_g[0]\) of chain complex of abelian groups becomes a distinguished triangle of \(D(\setZ)\) for all \(g \in G\).
    Since this sequence also comes from the stupid truncation of \(M_g \in \Ch^b(\setZ)\), the short exact sequence \(0 \to \sigma^{\geq 1}M_g^{\bullet} \to M_g \to M^0_g \to 0\) exists in \(\Ch^b(\setZ)\) and this gives the desired distinguished triangle on \(D(\setZ)\).

    The same argument works for the complex \(\grcomp{I}{M^{\bullet}}\) canonically and thus we have a commutative diagram
    \begin{center}
        \begin{tikzcd}
            \sigma^{\geq 1}M^{\bullet} \arrow[d] \arrow[r]   & M \arrow[d] \arrow[r]   & {M^0} \arrow[d]   \\
            \sigma^{\geq 1}\grcomp{I}{M^{\bullet}} \arrow[r] & \grcomp{I}{M^{\bullet}} \arrow[r] & {\grcomp{I}{M^0}}
        \end{tikzcd}
    \end{center}
    in \(\mcalD_{\graded{G}}(R)\) whose horizontal sequences are fiber sequences and vertical morphisms are the canonical morphism given in the first paragraph.
    By construction, we know \(\sigma^{\geq 1}\grcomp{I}{M^{\bullet}} = \grcomp{I}{\sigma^{\geq 1}M^{\bullet}}\).
    Then we have a commutative diagram
    \begin{center}
        \begin{tikzcd}
            \dgrcomp{I}{(\sigma^{\geq 1}M^{\bullet})} \arrow[d, "\cong"] \arrow[r]   & \dgrcomp{I}{M} \arrow[d] \arrow[r]   & \dgrcomp{I}{M^0} \arrow[d, "\cong"]   \\
            \sigma^{\geq 1}\grcomp{I}{M^{\bullet}} \arrow[r] & \grcomp{I}{M^{\bullet}} \arrow[r] & {\grcomp{I}{M^0}}
        \end{tikzcd}
    \end{center}
    in \(\mcalD_{\graded{G}}(R)\), where the isomorphism follows from the induction hypothesis and our assumption on each \(M^n\).
    Therefore, the above commutative diagram shows that the middle morphism \(\dgrcomp{I}{M} \to \grcomp{I}{M^{\bullet}}\) is an isomorphism.
\end{proof}

\section{Retraction datum} \label{SectionRetraction}

This section and the next section are devoted to the proof of \Cref{intro}. In this section, we prove some general results on adjunctions and mapping spaces in \(\infty\)-categories.

\begin{definition}
    Let \(\Ret\) be the category generated by morphisms \(i\colon0\to1\) and \(r\colon1\to0\), subject to the relation \(r\circ i=\id_0\).
    Let \(\mcalE\) be an \(\infty\)-category and let \(x\) and \(y\) are objects of \(\mcalE\).
    A \emph{retraction datum} from \(x\) to \(y\) is a functor $\Ret\to\mcalE$ which carries \(0\) to \(x\) and \(1\) to \(y\). When we fix a morphism \(\alpha \colon x \to y\) of \(\mcalE\), a retraction datum for \(\alpha\) is a retraction datum from \(x\) to \(y\) such that \(i\) in \(\Ret\) goes to \(\alpha\).
\end{definition}

\begin{lemma} \label{TrapezoidLemma}
    Let \(\mcalC\) be an idempotent complete \(\infty\)-category, let \(J\) be a simplicial set, and let \(X,Y\colon J\to\mcalC\) be diagrams. Fix a morphism \(\alpha \colon X \to Y\) in \(\Fun(J,\mcalC)\) and a retraction datum \(\overline{\alpha} \colon \Ret \to \Fun(J,\mcalC)\) for \(\alpha\).
    \begin{enumerate}
        \item If \(\lim_JY\) exists, then \(\lim_JX\) exists and the given retraction datum \(\overline{\alpha}\) induces a retraction datum for \(\lim_J \alpha \colon \lim_JX \to \lim_JY\).
        \item Let \(\mcalF\colon\mcalC\to\mcalD\) be a functor. If the limits of \(X\), \(Y\), \(\mcalF\circ X\), and \(\mcalF\circ Y\) exist, then there exists a retraction datum in \(\Fun(\Delta^1, \mcalD)\) from the canonical comparison morphism
        \begin{equation*}
            \mcalF(\lim_JX)\to\lim_J\mcalF(X)
        \end{equation*}
        to the canonical comparison morphism
        \begin{equation*}
            \mcalF(\lim_JY)\to\lim_J\mcalF(Y)
        \end{equation*}
        which induces retraction data in \(\mcalD\) for \(\mcalF(\lim_J \alpha)\) and \(\lim_J \mcalF(\alpha)\) on the source and target, respectively.
    \end{enumerate}
\end{lemma}

\begin{proof}
    For \(K\in\mcalC\), set
    \begin{equation*}
        P_X(K)\defeq\lim_{j\in J}\Map_{\mcalC}(K,X_j), \qquad P_Y(K)\defeq\lim_{j\in J}\Map_{\mcalC}(K,Y_j) \simeq \Map_{\mcalC}(K,\lim_JY)
    \end{equation*}
    in \(\Ani\). This defines presheaves \(P_X,P_Y\in\Fun(\mcalC^{\opposite},\Ani)\).

    The composition
    \begin{equation*}
        \Fun(J, \mcalC) \xrightarrow{\Fun(J, y)} \Fun(J, \Fun(\mcalC^{\opposite}, \Ani)) \xrightarrow{\lim} \Fun(\mcalC^{\opposite}, \Ani),
    \end{equation*}
    where \(y\colon\mcalC\to\Fun(\mcalC^{\opposite},\Ani)\) is the Yoneda embedding, carries \(X\) to \(P_X\) and \(Y\) to \(P_Y\). Composing the given retraction datum \(\overline{\alpha} \colon \Ret \to \Fun(J, \mcalC)\) with the above functor therefore induces a retraction datum from \(P_X\) to \(P_Y\) in \(\Fun(\mcalC^{\opposite}, \Ani)\).
    The presheaf \(P_Y\) is represented by \(\lim_JY\). By \citeKero{040X}, the closure of the Yoneda image under retracts is the idempotent completion of \(\mcalC\). Since \(\mcalC\) is idempotent complete, \(P_X\) is represented by an object \(X_\infty\in\mcalC\). The representing equivalence
    \begin{equation*}
        \Map_{\mcalC}(K,X_\infty)\simeq\lim_{j\in J}\Map_{\mcalC}(K,X_j)
    \end{equation*}
    is natural in \(K\), so \(X_\infty\) is a limit of \(X\). Since the Yoneda embedding is fully faithful, the retraction datum from \(P_X\) to \(P_Y\) is represented by a retraction datum for \(\lim_J \alpha \colon \lim_JX \to \lim_JY\). This proves (1).

    For (2), let \(\Fun(J,\mcalC)_{\lim}\) be the full subcategory spanned by the diagrams which admit limits. This gives a functor
    \begin{equation*}
        \lim_J\colon\Fun(J,\mcalC)_{\lim}\to\mcalC,
    \end{equation*}
    and similarly for \(\mcalD\). Applying \(\mcalF\) pointwise to the retraction datum \(\overline{\alpha}\) and applying these limit functors gives retraction data \(\mcalF(\lim_J \overline{\alpha})\) and \(\lim_J \mcalF(\overline{\alpha})\) between the sources and between the targets of the two comparison morphisms. The comparison morphism
    \begin{equation*}
        \mcalF(\lim_J(-))\to\lim_J(\mcalF\circ(-))
    \end{equation*}
    is natural on the full subcategory of \(\Fun(J,\mcalC)\) on which both sides are defined. Evaluating this natural transformation on the given \(\Ret\)-diagram therefore gives a functor
    \begin{equation*}
        \Ret\to\Fun(\Delta^1,\mcalD)
    \end{equation*}
    whose values at \(0\) and \(1\) are the two comparison morphisms \(\mcalF(\lim_J \alpha)\) and \(\lim_J \mcalF(\alpha)\). This proves (2).
\end{proof}

\begin{theorem} \label{GeneralSplitUnitTheorem}
    Let
    \begin{equation*}
        \mcalF\colon\mcalC\rightleftarrows\mcalD\colon\mcalG
    \end{equation*}
    be an adjunction of \(\infty\)-categories, let \(\eta\colon\id_{\mcalC}\to\mcalG\mcalF\) be the unit, and put \(\mcalT\defeq\mcalF\mcalG\). Assume that \(\mcalC\) is idempotent complete and that there is a retraction datum \(\overline{\eta}\) for \(\eta\) in \(\End(\mcalC)\). Then the following assertions hold.
    \begin{enumerate}
        \item The functor \(\mcalF\) is conservative.
        \item Let \(X\colon J\to\mcalC\) be a diagram. Assume that the limits \(\lim_{j \in J} \mcalF(X_j)\) and \(\lim_{j \in J} \mcalT(\mcalF(X_j))\) in \(\mcalD\) exist. Then \(\lim_JX\) exists in \(\mcalC\), and there exists a retraction datum in \(\Fun(\Delta^1, \mcalD)\) from the comparison morphism
        \begin{equation} \label{GeneralSplitUnitComparison}
            \mcalF(\lim_{j\in J}X_j) \to \lim_{j\in J}\mcalF(X_j)
        \end{equation}
        to the comparison morphism
        \begin{equation} \label{GeneralSplitUnitComonadComparison}
            \mcalF(\lim_{j \in J} \mcalG\mcalF(X_j)) \cong \mcalT(\lim_{j\in J}\mcalF(X_j))\to \lim_{j\in J}\mcalT(\mcalF(X_j)) = \lim_{j\in J}\mcalF\mcalG\mcalF(X_j)
        \end{equation}
        which induces retraction data for \(\mcalF(\lim_{j \in J} \eta(X_j))\) and \(\lim_{j \in J} \mcalF(\eta(X_j))\) on the source and target, respectively.
        In particular, if \eqref{GeneralSplitUnitComonadComparison} is an isomorphism, then \eqref{GeneralSplitUnitComparison} is an isomorphism.
    \end{enumerate}
\end{theorem}

\begin{proof}
    Let \(\sigma\colon\mcalG\mcalF\to\id_{\mcalC}\) be the retraction of \(\eta\) given by \(\overline{\eta}\). Naturality of \(\eta\) and \(\sigma\) gives a retraction datum in \(\Fun(\Delta^1,\mcalC)\) from every morphism \(f\) of \(\mcalC\) to \(\mcalG\mcalF(f)\). If \(\mcalF(f)\) is an equivalence, then \(\mcalG\mcalF(f)\) is an equivalence, and hence so is its retract \(f\). This proves (1).

    For (2), composing the functor \(X \colon J \to \mcalC\) with objects of \(\End(\mcalC)\) gives a retraction datum \(\overline{\eta} \circ X\) for \(\eta \circ X \colon X \to \mcalG\mcalF \circ X\) in \(\Fun(J,\mcalC)\). Since \(\mcalG\) is a right adjoint, it preserves the assumed limit \(\lim_J \mcalF(X_j)\), so
    \begin{equation*}
        \lim_J (\mcalG\mcalF \circ X) \simeq \lim_{j\in J}\mcalG\mcalF(X_j)\simeq\mcalG(\lim_{j\in J}\mcalF(X_j))
    \end{equation*}
    exists in \(\mcalC\) and we have a retraction datum for \(\lim_J (\eta \circ X)\).
    By part~(1) of \Cref{TrapezoidLemma}, the limit \(\lim_JX\) exists. Applying part~(2) of \Cref{TrapezoidLemma} to \(\mcalF\) and the retraction datum \(\overline{\eta} \circ X\) for \(\eta \circ X\), there exists a retraction datum from the comparison morphism \eqref{GeneralSplitUnitComparison} to
    \begin{equation*}
        \mcalF(\lim_J (\mcalG\mcalF \circ X)) \simeq \mcalF\mcalG(\lim_{j\in J}\mcalF(X_j))\to\lim_{j\in J}\mcalF\mcalG\mcalF(X_j) \simeq \lim_J \mcalF(\mcalG\mcalF \circ X)
    \end{equation*}
    which is \eqref{GeneralSplitUnitComonadComparison} since \(\mcalG\) preserves the limit \(\lim_J\mcalF(X_j)\).
    The last assertion follows from \citeKero{03YP}.
\end{proof}

\begin{definition}
    Let \(\mcalC\) be an \(\infty\)-category. Following \citeKero{04S9}, an augmented simplicial object of \(\mcalC\) is called \emph{split} if it admits a splitting in the sense defined there, and a simplicial object is called split if it extends to a split augmented simplicial object. An augmented cosimplicial object of \(\mcalC\) is called \emph{split} if the corresponding augmented simplicial object of \(\mcalC^{\opposite}\) is split, and a cosimplicial object is called split if it extends to a split augmented cosimplicial object.
\end{definition}

\begin{lemma} \label{AbsoluteSplitTotalization}
    Let \(X_\bullet\to X_{-1}\) be a split augmented simplicial object of an \(\infty\)-category \(\mcalC\). Then the augmentation exhibits \(X_{-1}\) as the geometric realization of \(X_\bullet\), and this colimit is preserved by every functor defined on \(\mcalC\). Dually, if \(X^{-1}\to X^\bullet\) is a split augmented cosimplicial object, then the augmentation exhibits \(X^{-1}\) as the limit of \(X^\bullet\), and this limit is preserved by every functor defined on \(\mcalC\).
\end{lemma}

\begin{proof}
    The simplicial assertion is \citeKero{04SM}; its formulation uses the notion of splitting from \citeKero{04S9}. Applying \citeKero{04SM} to \(\mcalC^{\opposite}\) gives the cosimplicial assertion.
\end{proof}

\begin{corollary} \label{SplitUnitComonadicity}
    In the setting of \Cref{GeneralSplitUnitTheorem}, the functor \(\mcalF\) is comonadic.
\end{corollary}

\begin{proof}
    Apply the comonadic Barr--Beck--Lurie theorem, namely the opposite form of \cite{lurie2017Higher}*{Proposition 4.7.3.3 and Theorem 4.7.3.5}. By part~(1) of \Cref{GeneralSplitUnitTheorem}, the functor \(\mcalF\) is conservative. Let \(X^\bullet\colon\Delta\to\mcalC\) be a cosimplicial object such that \(\mcalF(X^\bullet)\) is split, and choose a split augmentation \(Y\to\mcalF(X^\bullet)\). By \Cref{AbsoluteSplitTotalization},
    \begin{equation*}
        Y\simeq\lim_\Delta\mcalF(X^\bullet)
    \end{equation*}
    and this limit is preserved by \(\mcalT=\mcalF\mcalG\) because of the splitting of \(\mcalF(X^{\bullet})\). Hence
    \begin{equation*}
        \mcalT(Y)\xrightarrow{\simeq}\lim_\Delta\mcalT(\mcalF(X^\bullet)).
    \end{equation*}
    By part~(2) of \Cref{GeneralSplitUnitTheorem}, the limit \(\lim_\Delta X^\bullet\) exists and is preserved by \(\mcalF\). The Barr--Beck--Lurie theorem therefore applies.
\end{proof}

\begin{construction} \label{AdjunctionBarConstruction}
    Let
    \begin{equation} \label{AdjunctionForBarConstruction}
        F\colon\mcalC\rightleftarrows\mcalD\colon U
    \end{equation}
    be an adjunction, with unit \(\eta\colon\id_{\mcalC}\to UF\) and counit \(\varepsilon\colon FU\to\id_{\mcalD}\). For \(M\in\mcalD\), the monadic bar construction of \cite{lurie2017Higher}*{Example 4.7.2.7} is the functorial augmented simplicial object
    \begin{equation*}
        B_\bullet^F(M)\to M
    \end{equation*}
    in \(\mcalD\) with
    \begin{equation} \label{AdjunctionBarTerms}
        B_n^F(M)\defeq(FU)^{n+1}(M)\simeq F((UF)^nU(M)) \qquad (n\geq0).
    \end{equation}
    Its face morphisms are induced by the counit and its degeneracy morphisms are induced by the unit. The augmented simplicial object \(U(B_\bullet^F(M))\to U(M)\) is split.
\end{construction}

\begin{lemma} \label{MonadBarResolution}
    Keep the notation of \Cref{AdjunctionBarConstruction}. Assume that \(\lvert B_\bullet^F(M)\rvert\) exists in \(\mcalD\), that \(U\) preserves this geometric realization, and that \(U\) is conservative. Then the augmentation induces an isomorphism
    \begin{equation*}
        \lvert B_\bullet^F(M)\rvert\xrightarrow{\cong}M.
    \end{equation*}
\end{lemma}

\begin{proof}
    By \Cref{AdjunctionBarConstruction}, the augmented simplicial object \(U(B_\bullet^F(M))\to U(M)\) is split. Hence \Cref{AbsoluteSplitTotalization} identifies its geometric realization with \(U(M)\). Since \(U\) preserves the geometric realization and is conservative, the augmentation is an isomorphism.
\end{proof}

\begin{proposition} \label{MonadMappingSpaceTotalization}
    In the setting of \Cref{MonadBarResolution}, take \(M,N\in\mcalD\) and assume that the hypotheses of that lemma hold for \(M\). Then there is a functorial equivalence
    \begin{equation} \label{MonadMappingSpaceTotalizationEquation}
        \Map_{\mcalD}(M,N)\xrightarrow{\simeq}\lim_{[n]\in\Delta}\Map_{\mcalC}((UF)^nU(M),U(N)),
    \end{equation}
    where the cosimplicial structure on the right is induced from the full simplicial bar construction.
\end{proposition}

\begin{proof}
    By \Cref{MonadBarResolution}, the augmentation \(B_\bullet^F(M)\to M\) is a geometric realization diagram. Hence
    \begin{equation*}
        \Map_{\mcalD}(M,N)\simeq\lim_{[n]\in\Delta}\Map_{\mcalD}(B_n^F(M),N).
    \end{equation*}
    The adjunction \eqref{AdjunctionForBarConstruction} and \eqref{AdjunctionBarTerms} identify the \(n\)-th term on the right with
    \begin{equation*}
        \Map_{\mcalC}((UF)^nU(M),U(N)).
    \end{equation*}
    These equivalences are natural in the simplicial variable and therefore identify the full cosimplicial objects.
\end{proof}

\begin{construction} \label{BarMappingCosimplicialObject}
    Let \(\mcalV\) be a presentable monoidal \(\infty\)-category whose tensor product preserves colimits separately in each variable, let \(\mcalM\) be a presentable \(\infty\)-category left-tensored over \(\mcalV\) such that the action preserves colimits separately in each variable, and let \(A\in\Alg(\mcalV)\). By \cite{lurie2017Higher}*{Corollaries 4.2.4.6 and 4.2.4.8}, the forgetful functor
    \begin{equation*}
        U_A\colon\Mod_A(\mcalM)\to\mcalM
    \end{equation*}
    admits the free module functor \(F_A(X)=A\otimes X\) as a left adjoint. The functor \(U_A\) is conservative by \cite{lurie2017Higher}*{Corollary 4.2.3.2}, and geometric realizations in \(\Mod_A(\mcalM)\) exist and are preserved by \(U_A\) by \cite{lurie2017Higher}*{Corollary 4.2.3.5}.

    For \(M,N\in\Mod_A(\mcalM)\), define the \emph{bar mapping cosimplicial space} by
    \begin{equation*}
        C_A^\bullet(M,N)\defeq\Map_{\Mod_A(\mcalM)}(B_\bullet^{F_A}(M),N).
    \end{equation*}
    The adjunction gives functorial equivalences
    \begin{equation} \label{AlgebraBarMappingCosimplicialTerms}
        C_A^n(M,N)\simeq\Map_{\mcalM}(A^{\otimes n}\otimes U_A(M),U_A(N)) \qquad (n\geq0).
    \end{equation}
    Here \(A^{\otimes0}\otimes U_A(M)\defeq U_A(M)\). The two coface morphisms are induced by the \(A\)-actions on \(M\) and \(N\), the other coface morphisms are induced by multiplication in \(A\), and the codegeneracy morphisms are induced by the unit of \(A\).
\end{construction}

\begin{corollary} \label{ModuleMappingSpaceTotalization}
    In the situation of \Cref{BarMappingCosimplicialObject}, there is a functorial equivalence
    \begin{equation} \label{AlgebraModuleMappingSpaceTotalization}
        \Map_{\Mod_A(\mcalM)}(M,N)\xrightarrow{\simeq}\lim_{[n]\in\Delta}\Map_{\mcalM}(A^{\otimes n}\otimes U_A(M),U_A(N)).
    \end{equation}
    In particular, let \(U_{\gr}\colon\mcalD_{\graded{G}}(R)\to\mcalD_{\graded{G}}(\setZ)\) and \(U\colon\mcalD(R)\to\mcalD(\setZ)\) be the forgetful functors. For \(M,N\in\mcalD_{\graded{G}}(R)\), there is a functorial equivalence
    \begin{equation} \label{GradedModuleMappingSpaceTotalization}
        \Map_{\mcalD_{\graded{G}}(R)}(M,N)\xrightarrow{\simeq}\lim_{[n]\in\Delta}\Map_{\mcalD_{\graded{G}}(\setZ)}\bigl(R^{\Lgrotimes_{\setZ}n}\Lgrotimes_{\setZ}U_{\gr}(M),U_{\gr}(N)\bigr).
    \end{equation}
    Likewise, for \(M,N\in\mcalD(R)\), there is a functorial equivalence
    \begin{equation} \label{UngradedModuleMappingSpaceTotalization}
        \Map_{\mcalD(R)}(M,N)\xrightarrow{\simeq}\lim_{[n]\in\Delta}\Map_{\mcalD(\setZ)}\bigl(R^{\otimes^L_{\setZ}n}\otimes^L_{\setZ}U(M),U(N)\bigr).
    \end{equation}
\end{corollary}

\begin{proof}
    Apply \Cref{MonadMappingSpaceTotalization} to the free--forgetful adjunction \(F_A\dashv U_A\). The graded and ungraded formulas follow by applying the first assertion to the algebra object \(R\) in \(\mcalD_{\graded{G}}(\setZ)\) and in \(\mcalD(\setZ)\), respectively.
\end{proof}

\section{Graded modules and formal comodules} \label{SectionBBL}

We construct a retraction datum for a unit for the graded forgetful adjunction and its completed version, and apply the general results of \Cref{SectionRetraction} to obtain the limit comparison and the comonadic equivalence.

\begin{definition} \label{DefMG}
    Let \(G\) be an abelian group, let \(R\) be a \(G\)-graded ring, and let \(I\) be a finitely generated homogeneous ideal of \(R\).
    Write
    \begin{equation} \label{ForgetfulCompletionFunctorGradedModules}
        \mcalF^0\colon\mcalD_{\graded{G}}(R)\to\mcalD(R), \qquad \mcalF^I\colon\mcalD_{\graded{G}}^{\comp{I}}(R)\to\mcalD^{\comp{I}}(R)
    \end{equation}
    for the forgetful functor and the functor \(M\mapsto\dcomp{I}{\mcalF^0(M)}\), respectively.
    Let \(\rho_R\colon R\to R[G]\) be the grading coaction of \Cref{def-coaction}.
    For \(N\in\mcalD(R)\), define
    \begin{equation} \label{DefMGEquation}
        N[G]\defeq\rho_{R,*}(N\otimes_R^LR[G]) \in \mcalD(R),
    \end{equation}
    where the tensor product uses the canonical morphism \(R\to R[G]\), \(r\mapsto rt^0\), and \(\rho_{R,*}\) is restriction of scalars along \(\rho_R\).
    We retain this notation and these assumptions throughout this section.
\end{definition}

\begin{proposition} \label{FactorizationMG}
    Let \(R\) be a \(G\)-graded ring. Then the forgetful functor\footnote{The symbol \((-)^0\) indicates the adjunctions for non-completed objects, i.e., \(0\)-complete objects.}
    \begin{equation*}
        \mcalF^0 \colon \mcalD_{\graded{G}}(R) \to \mcalD(R)
    \end{equation*}
    admits a right adjoint \(\mcalG^0 \colon \mcalD(R) \to \mcalD_{\graded{G}}(R)\) such that the functor defined in \Cref{DefMG} admits a factorization
    \begin{center}
        \begin{tikzcd}
            & & \mcalD(R). \\
            \mcalD(R) \arrow[rru, "{M \mapsto M[G]}"] \arrow[rr, "\mcalG^0"] \arrow[rrd, "M \mapsto \bigoplus_{g \in G} M \cdot t^g"'] & & \mcalD_{\graded{G}}(R) \arrow[u, "\mcalF^0"'] \arrow[d, "\mathrm{forget}"] \\
            & & \mcalD_{\graded{G}}(\setZ).
        \end{tikzcd}
    \end{center}
    Moreover, if \(\eta^0 \colon \id \to \mcalG^0\mcalF^0\) denotes the unit, then \(\eta^0\) admits a retraction datum
    \begin{equation*}
        \overline{\eta}^0 \colon \Ret \to \Fun(\mcalD_{\graded{G}}(R),\mcalD_{\graded{G}}(R)).
    \end{equation*}
    More precisely, there is a natural equivalence
    \begin{equation} \label{NonCompletedShiftDecomposition}
        \mcalG^0\mcalF^0(M) \xrightarrow{\simeq} \bigoplus_{d \in G} M(d)
    \end{equation}
    under which \(\eta^0_M = \overline{\eta}^0(s \colon 0 \to 1)_M\) is the inclusion of the summand indexed by \(d=0\), and \(\overline{\eta}^0(r \colon 1 \to 0)_M\) is the corresponding projection.
\end{proposition}

\begin{proof}
    Let \(\mcalF_0 \colon \Mod_{\graded{G}}(R) \to \Mod(R)\) be the forgetful functor. For \(N \in \Mod(R)\), define \(\mcalG_0(N) \in \Mod_{\graded{G}}(R)\) by \((\mcalG_0(N))_g=N\) for every \(g\in G\), with \(R_h\) carrying the component of degree \(g\) to that of degree \(g+h\) through the given action of \(R_h\) on \(N\). Equivalently,
    \begin{equation*}
        \mcalG_0(N)=\rho_{R,*}(N\otimes_RR[G]) \in \Mod_{\graded{G}}(R).
    \end{equation*}
    For \(M \in \Mod_{\graded{G}}(R)\), a graded \(R\)-linear morphism \(M\to\mcalG_0(N)\) is the same as a family of maps \(M_g\to N\) whose direct sum \(\mcalF_0(M) = \bigoplus_{g \in G} M_g \to N\) is \(R\)-linear. Hence there are natural bijections
    \begin{equation} \label{DiscreteGradedAdjunction}
        \operatorname{Hom}_{\Mod_{\graded{G}}(R)}(M,\mcalG_0(N)) \cong \operatorname{Hom}_{\Mod(R)}(\mcalF_0(M),N),
    \end{equation}
    and therefore \(\mcalF_0\dashv\mcalG_0\), with unit \(\eta_0\colon\id\to\mcalG_0\mcalF_0\) whose degree-\(g\) component at \(M\) is the natural inclusion \(M_g\hookrightarrow\mcalF_0(M)\), corresponding to \(\id_{\mcalF_0(M)}\) under \eqref{DiscreteGradedAdjunction}.
    Also note that both functors are exact.

    For \(M \in \Mod_{\graded{G}}(R)\), reindexing by \(h=g+d\) in degree \(g\) gives a natural isomorphism of graded \(R\)-modules
    \begin{equation} \label{DiscreteShiftDecomposition}
        \alpha_M\colon\bigoplus_{d\in G}M(d)\xrightarrow{\cong}\mcalG_0\mcalF_0(M):
    \end{equation}
    Indeed, we have isomorphisms of graded \(\setZ\)-modules
    \begin{equation*}
        \mcalG_0\mcalF_0(M) = \bigoplus_{g \in G} \bigoplus_{h \in G} M_h \cdot t^g \cong \bigoplus_{d \in G} \bigoplus_{h-g = d} M_h \cdot t^g \cong \bigoplus_{d \in G} \bigoplus_{g \in G} (M(d))_g \cong \bigoplus_{d \in G} M(d).
    \end{equation*}
    For any \(k,d,g\in G\), set \(h=g+d\). Multiplication by \(r_k\in R_k\) carries \(M_h\cdot t^g\) to \(M_{h+k}\cdot t^{g+k}\), corresponding to the map \(M(d)_g\to M(d)_{g+k}\). Thus the reindexing preserves the \(R\)-action and is an isomorphism of graded \(R\)-modules.
    
    Under \eqref{DiscreteShiftDecomposition}, the unit \(\eta_{0,M}\colon M\to\mcalG_0\mcalF_0(M)\) of \eqref{DiscreteGradedAdjunction} is the inclusion of the summand indexed by \(d=0\) since it is induced from the inclusion \(M_g \hookrightarrow M \cong M \cdot t^g\) on degree \(g\) and this corresponds to \(M_g = (M(0))_g \hookrightarrow M(0)\).
    Let \(\sigma_{0,M}\) be the corresponding projection. These maps are natural in \(M\) and satisfy \(\sigma_0\eta_0=\id\) strictly. Thus they define a functor
    \begin{equation} \label{DiscreteRetractionDatum}
        \overline{\eta}_0\colon\Ret\to\Fun(\Mod_{\graded{G}}(R),\Mod_{\graded{G}}(R))
    \end{equation}
    with \(\overline{\eta}_0(0)=\id\), \(\overline{\eta}_0(1)=\mcalG_0\mcalF_0\), \(\overline{\eta}_0(s)=\eta_0\), and \(\overline{\eta}_0(r)=\sigma_0\).

    We pass to derived \(\infty\)-categories. The abelian categories \(\Mod_{\graded{G}}(R)\) and \(\Mod(R)\) are Grothendieck; for the former, this is recalled in \Cref{DerivedGradedCatEquiv}. By \cite{lurie2017Higher}*{Proposition 1.3.5.3}, their categories of chain complexes carry the injective model structures, with quasi-isomorphisms as weak equivalences and degreewise monomorphisms as cofibrations. Applying \(\mcalF_0\) and \(\mcalG_0\) degreewise gives an adjunction
    \begin{equation*}
        \Ch(\mcalF_0)\colon\Ch(\Mod_{\graded{G}}(R))\rightleftarrows\Ch(\Mod(R))\colon\Ch(\mcalG_0).
    \end{equation*}
    Since \(\mcalF_0\) is exact, \(\Ch(\mcalF_0)\) preserves cofibrations and trivial cofibrations, so this is a Quillen adjunction. The underlying \(\infty\)-categories are \(\mcalD(\Mod_{\graded{G}}(R))\) and \(\mcalD(\Mod(R))\) by \cite{lurie2017Higher}*{Proposition 1.3.5.15}; hence \cite{mazel-gee2016Quillen}*{Theorem 2.1} gives an adjunction
    \begin{equation} \label{DerivedDiscreteAdjunction}
        \overline{\mcalF}_0\colon\mcalD(\Mod_{\graded{G}}(R))\rightleftarrows\mcalD(\Mod(R))\colon\overline{\mcalG}_0.
    \end{equation}
    Since \(\mcalF_0\) and \(\mcalG_0\) are exact, their degreewise extensions preserve quasi-isomorphisms. Thus \(\overline{\mcalF}_0\) and \(\overline{\mcalG}_0\) are canonically identified with the functors induced by \(\Ch(\mcalF_0)\) and \(\Ch(\mcalG_0)\) on the \(\infty\)-categorical localizations, by the universal property in \cite{lurie2017Higher}*{Definition 1.3.4.1 and Remark 1.3.4.2}.

    We next localize the entire retraction datum. Write \(\eta_0\) and \(\varepsilon_0\) also for the unit and counit of the adjunction on chain complexes, and let
    \begin{equation*}
        q\colon\Ch(\Mod_{\graded{G}}(R))\to\mcalD(\Mod_{\graded{G}}(R))
    \end{equation*}
    be the localization. Applying \eqref{DiscreteRetractionDatum} degreewise and then postcomposing with \(q\) gives a functor
    \begin{equation} \label{RetBeforeLocalization}
        \Ret\to\Fun(\Ch(\Mod_{\graded{G}}(R)),\mcalD(\Mod_{\graded{G}}(R))).
    \end{equation}
    It sends \(0\) to \(q\), \(1\) to \(q\circ\Ch(\mcalG_0)\Ch(\mcalF_0)\), \(s\) to \(q\eta_0\), and \(r\) to \(q\sigma_0\).

    By the universal property of localization in \cite{lurie2017Higher}*{Definition 1.3.4.1 and Remark 1.3.4.2}, precomposition with \(q\) gives a fully faithful functor
    \begin{equation*}
        q^*\colon\Fun(\mcalD(\Mod_{\graded{G}}(R)),\mcalD(\Mod_{\graded{G}}(R)))\longrightarrow\Fun(\Ch(\Mod_{\graded{G}}(R)),\mcalD(\Mod_{\graded{G}}(R)))
    \end{equation*}
    whose essential image consists of the functors carrying quasi-isomorphisms to equivalences. Both vertices of the retraction datum in \eqref{RetBeforeLocalization} belong to this full subcategory. Therefore this functor from \(\Ret\) lifts along \(q^*\), uniquely up to a contractible space of choices, to a functor
    \begin{equation} \label{DerivedRetractionDatumBeforeEquivalence}
        \Ret\to\Fun(\mcalD(\Mod_{\graded{G}}(R)),\mcalD(\Mod_{\graded{G}}(R)))
    \end{equation}
    with vertices \(\id\) and \(\overline{\mcalG}_0\overline{\mcalF}_0\). Write
    \begin{equation*}    \widetilde{\eta}\colon\id_{\mcalD(\Mod_{\graded{G}}(R))}\to\overline{\mcalG}_0\overline{\mcalF}_0
    \end{equation*}
    for its morphism over \(s\colon0\to1\) in \(\Ret\). Since \eqref{DerivedRetractionDatumBeforeEquivalence} lifts \eqref{RetBeforeLocalization}, it satisfies
    \begin{equation} \label{FormulaqEtaTilde}
        q^*(\widetilde{\eta})=\widetilde{\eta}q\simeq\bigl(q\xrightarrow{q\eta_0}q\Ch(\mcalG_0)\Ch(\mcalF_0)\simeq\overline{\mcalG}_0\overline{\mcalF}_0q\bigr),
    \end{equation}
    and the morphism over \(r \colon 1 \to 0\) similarly restricts to \(q\sigma_0 \colon q\Ch(\mcalG_0)\Ch(\mcalF_0) \to q\).

    We claim that \(\widetilde{\eta}\) is a unit for \(\overline{\mcalF}_0\dashv\overline{\mcalG}_0\). Since both \(\Ch(\mcalF_0)\) and \(\Ch(\mcalG_0)\) preserve quasi-isomorphisms, the adjunction \(\Ch(\mcalF_0) \dashv \Ch(\mcalG_0)\) with triangle identities of \(\eta_0\) and \(\varepsilon_0\) induces the adjunction on the ordinary localizations of the categories of chain complexes, namely, on the \(1\)-categorical derived categories \(D(\Mod_{\graded{G}}(R))\) and \(D(\Mod(R))\).
    The composition with \(q\) and the natural functor induced from taking the homotopy category is equivalent to the canonical functor from the category of chain complexes to the \(1\)-categorical derived categories.
    Therefore, the induced natural transformation \(h\widetilde{\eta} \colon \id_{h\mcalD(\Mod_{\graded{G}}(R))} \to h\overline{\mcalG}_0\overline{\mcalF}_0\) on the homotopy categories is equivalent to the image of \(\eta_0\).
    The functor \(\overline{\mcalF}_0\) already admits a right adjoint by \eqref{DerivedDiscreteAdjunction}, so \citeKero{02EY} implies that \(\widetilde{\eta}\) is a unit for the adjunction of \(\infty\)-categories.
    
    We use this unit for \eqref{DerivedDiscreteAdjunction}. Consequently, \eqref{DerivedRetractionDatumBeforeEquivalence} is a retraction datum for its unit.
    The same localization sends \eqref{DiscreteShiftDecomposition} to a natural equivalence
    \begin{equation} \label{DerivedShiftDecompositionBeforeEquivalence}
        \overline{\mcalG}_0\overline{\mcalF}_0(M)\simeq\bigoplus_{d\in G}M(d)
    \end{equation}
    in \(\mcalD(\Mod_{\graded{G}}(R))\), under which the unit and its retraction are the inclusion and projection of the summand \(d=0\).

    Finally, \Cref{DerivedGradedCatEquiv} gives a \(t\)-exact equivalence
    \begin{equation*}
        \mcalD(\Mod_{\graded{G}}(R))\xrightarrow{\simeq}\mcalD_{\graded{G}}(R)
    \end{equation*}
    compatible with the forgetful functors to \(\mcalD(\Mod(R))\simeq\mcalD(R)\). Transporting \eqref{DerivedDiscreteAdjunction} along this equivalence identifies \(\overline{\mcalF}_0\) with \(\mcalF^0\); write \(\mcalG^0\) for the transported right adjoint. Transporting \eqref{DerivedRetractionDatumBeforeEquivalence} and \eqref{DerivedShiftDecompositionBeforeEquivalence} gives the required retraction datum \(\overline{\eta}^0\) and the equivalence \eqref{NonCompletedShiftDecomposition}.

    It remains only to identify \(\mcalG^0\) with the functor of \Cref{DefMG}. Before localization, there is a natural isomorphism of complexes of graded \(R\)-modules
    \begin{equation} \label{ChainCofreeGroupRingIdentification}
        \Ch(\mcalG_0)(N^\bullet) \xrightarrow{\cong} \rho_{R,*}(N^\bullet \otimes_R R[G]).
    \end{equation}
    Indeed, both sides have degree-\(g\) component \(N^\bullet\cdot t^g\), and under this identification an element of \(R_h\) acts by \(nt^g\mapsto(rn)t^{g+h}\). The coefficient map \(R\to R[G]\) makes \(R[G]=\bigoplus_{g\in G}R\cdot t^g\) a free \(R\)-module. Hence, after forgetting the grading, localization of \eqref{ChainCofreeGroupRingIdentification} gives a natural equivalence
    \begin{equation*}
        \mcalF^0\mcalG^0(N) \xrightarrow{\simeq} \rho_{R,*}(N\otimes_R^LR[G])=N[G]
    \end{equation*}
    in \(\mcalD(R)\), proving the upper composite in the displayed factorization. Forgetting the \(R\)-module structure in \eqref{ChainCofreeGroupRingIdentification} and localizing gives the graded object \(\bigoplus_{g\in G}N\cdot t^g\), proving the lower composite. This proves the factorization and completes the proof.
\end{proof}

\begin{proposition} \label{NonCompletedSplitUnit}
    The functor \(\mcalF^0\) admits a right adjoint \(\mcalG^0\).
    The underlying \(R\)-module of \(\mcalG^0(N)\) is \(N[G]\), and its component of degree \(g\) is \(N\cdot t^g\).
    Multiplication by \(R_h\) carries \(N\cdot t^g\) to \(N\cdot t^{g+h}\) through the ordinary \(R\)-action on \(N\).
    There is a functorial equivalence of graded \(R\)-modules
    \begin{equation} \label{CofreeForgetfulShiftDecomposition}
        \mcalG^0\mcalF^0(M)\simeq\bigoplus_{d\in G}M(d),
    \end{equation}
    under which the unit \(\eta^0 \colon \id_{\mcalD_{\graded{G}}(R)} \to \mcalG^0\mcalF^0\) is the inclusion of the summand indexed by \(d=0\).
    In particular, there exists a retraction datum for \(\eta^0\).
\end{proposition}

\begin{proof}
    By \Cref{PropertiesDerivedGradedModules} and \cite{lurie2009Higher}*{Corollary 5.5.2.9}, the colimit-preserving functor \(\mcalF^0\) between presentable \(\infty\)-categories admits a right adjoint \(\mcalG^0\).
    Write \(\eta^0\) for its unit.
    For \(K\in\mcalD(\setZ)\), regarded as a trivially graded object, the free module adjunction gives
    \begin{equation*}
        \Map_{\mcalD(\setZ)}(K,(\mcalG^0(N))_g) \simeq \Map_{\mcalD_{\graded{G}}(\setZ)}(K(-g), \mcalG^0(N)) \simeq \Map_{\mcalD_{\graded{G}}(R)}(R\Lgrotimes_{\setZ}K(-g),\mcalG^0(N)).
    \end{equation*}
    By adjunction \(\mcalF^0 \dashv \mcalG^0\), the latter space is
    \begin{equation*}
        \Map_{\mcalD(R)}(R\otimes^L_{\setZ}K,N)\simeq\Map_{\mcalD(\setZ)}(K,N).
    \end{equation*}
    Yoneda embedding therefore identifies \((\mcalG^0(N))_g\) with the underlying object of \(N\) in \(\mcalD(\setZ)\).
    Under this identification, \((\eta_M^0)_g\) is the inclusion \(M_g\to\bigoplus_{h \in G}M_h\): the adjoint of \(\eta_M^0\) is \(\id_{\mcalF^0(M)}\), and applying \(\mcalF^0\) to a free module map \(R\Lgrotimes_{\setZ}K(-g)\to M\) gives the map determined by \(K\to M_g\to\bigoplus_{h \in G} M_h\).

    Let \(R^{\mathrm{triv}}\) denote the underlying ring of \(R\) with trivial grading, and let \(N^{\mathrm{triv}}\) be \(N\) with trivial grading.
    Equip \(R[G]\) with the coefficient grading \(R[G]_g=R\cdot t^g\).
    The coefficient inclusion \(R^{\mathrm{triv}}\to R[G]\) and the graded algebra morphism \(\rho_R:R\to R[G]\) define
    \begin{equation*}
        \widetilde{\mcalG}(N)\defeq\rho_{R,*}\bigl(R[G]\Lgrotimes_{R^{\mathrm{triv}}}N^{\mathrm{triv}}\bigr)\in\mcalD_{\graded{G}}(R).
    \end{equation*}
    Its underlying \(R\)-module is \(N[G]\), and it has the stated homogeneous action.
    The augmentation \(e:R[G]\to R\), \(t^g\mapsto1\), induces a morphism \(\mcalF^0\widetilde{\mcalG}(N)\to N\) in \(\mcalD(R)\), by extension of scalars along \(e\) and the identity \(e\rho_R=\id_R\).
    Its adjoint
    \begin{equation*}
        \widetilde{\mcalG}(N)\longrightarrow\mcalG^0(N) \in \mcalD_{\graded{G}}(R)
    \end{equation*}
    has degree \(g\) component \(N\to\bigoplus_{h \in G}N\to N\), where the maps are inclusion of the \(g\)-th summand and summation.
    Each component is the identity, so the morphism is an isomorphism in \(\mcalD_{\graded{G}}(R)\) by \Cref{PropertiesDerivedGradedModules}\Cref{LimitsColimitsDgrR}.
    This identifies the right adjoint \(\mcalG^0(N)\) with the asserted module \(N[G]\).

    Since \(M(d)\simeq M\Lgrotimes_R R(d)\) and \(\mcalF^0(R(d))\simeq R\), the symmetric monoidality of \(\mcalF^0\) gives a natural equivalence
    \begin{equation*}
        q_d:\mcalF^0((-)(d))\xrightarrow{\simeq}\mcalF^0
    \end{equation*}
    of functors \(\mcalD_{\graded{G}}(R)\to\mcalD(R)\), induced by the canonical identification of the underlying object of \(\setZ(d)\) with \(\setZ\).
    Take its adjoint natural transformation
    \begin{equation*}
        \alpha_d:(-)(d)\xrightarrow{\eta^0}\mcalG^0\mcalF^0((-)(d))\xrightarrow{\mcalG^0(q_d)}\mcalG^0\mcalF^0
    \end{equation*}
    in \(\End(\mcalD_{\graded{G}}(R))\).
    The coproduct of these transformations gives a morphism in the same category,
    \begin{equation*}
        \alpha:\bigoplus_{d\in G}(-)(d)\longrightarrow\mcalG^0\mcalF^0.
    \end{equation*}
    By the description of the unit above, its degree \(g\) component at \(M\) is the canonical reindexing
    \begin{equation*}
        \bigoplus_{d\in G}M_{g+d}\xrightarrow{\simeq}\bigoplus_{h\in G}M_h
    \end{equation*}
    in \(\mcalD(\setZ)\).
    Hence \(\alpha\) is an equivalence in \(\mcalD_{\graded{G}}(R)\) by conservativity, proving \eqref{CofreeForgetfulShiftDecomposition}.

    For \(d=0\), the transformation \(q_0\) is the identity and \(\alpha_0=\eta^0\).
    Thus \(\alpha\) identifies the unit with the inclusion of the \(d=0\) summand.
    Transporting the canonical inclusion and projection of the coproduct \(\id\oplus\bigoplus_{d\neq0}(-)(d)\) along \(\alpha\) gives the required retraction datum.
\end{proof}

\begin{proposition} \label{CofreeCompletionCompatibility}
    Completion induces a functorial equivalence
    \begin{equation} \label{CofreeCompletionCompatibilityEquation}
        \chi_N\colon\dgrcomp{I}{\mcalG^0(N)}\xrightarrow{\simeq}\mcalG^0(\dcomp{I}{N}).
    \end{equation}
    Consequently, the restriction \(\mcalG^I\defeq\mcalG^0|_{\mcalD^{\comp{I}}(R)}\) takes values in \(\mcalD_{\graded{G}}^{\comp{I}}(R)\) and is right adjoint to \(\mcalF^I\), and there exists a retraction datum for the unit \(\eta^I \colon \id_{\mcalD_{\graded{G}}^{\comp{I}}(R)} \to \mcalG^I\mcalF^I\).
\end{proposition}

\begin{proof}
    Let \(f\in R_d\) be homogeneous.
    By the description of \(\mcalG^0\) in \Cref{NonCompletedSplitUnit}, multiplication by \(t^d\) gives
    \begin{equation} \label{CofreeShiftEquivalence}
        \tau_{d,N}:\mcalG^0(N)(-d)\xrightarrow{\simeq}\mcalG^0(N).
    \end{equation}
    Since \(\rho_R(f)=ft^d\), graded multiplication by \(f\) is \(\mcalG^0(f)\circ\tau_{d,N}\).
    Exactness of \(\mcalG^0\) and \eqref{DerivedQuotCofiber} therefore give
    \begin{equation*}
        \mcalG^0(N)/^L(f)\simeq\cofib(\mcalG^0(f))\simeq\mcalG^0(N/^L(f)).
    \end{equation*}
    Iterating this argument, we obtain an equivalence of inverse systems
    \begin{equation} \label{CofreeCommutesDerivedQuotient}
        \mcalG^0(N/^L(f_1^n,\ldots,f_r^n))\simeq\mcalG^0(N)/^L(f_1^n,\ldots,f_r^n).
    \end{equation}
    Since \(\mcalG^0\) preserves limits, the isomorphism of \Cref{PrincipalDerivedGradedwiseCompProp}\Cref{DerivedGrcompQuotLimit}, and its specialization to the trivial grading group, give
    \begin{equation*}
        \dgrcomp{I}{\mcalG^0(N)}\simeq\grlim_{n\geq1}\mcalG^0(N/^L(f_1^n,\ldots,f_r^n))\simeq\mcalG^0(\dcomp{I}{N}).
    \end{equation*}
    This identifies \(\mcalG^0(N)\to\mcalG^0(\dcomp{I}{N})\) with the completion morphism and proves \eqref{CofreeCompletionCompatibilityEquation}.

    For complete objects \(M\in\mcalD_{\graded{G}}^{\comp{I}}(R)\) and \(N\in\mcalD^{\comp{I}}(R)\), fullness and the universal property of completion give
    \begin{equation*}
        \Map_{\mcalD_{\graded{G}}^{\comp{I}}(R)}(M,\mcalG^I(N))\simeq\Map_{\mcalD(R)}(\mcalF^0(M),N)\simeq\Map_{\mcalD^{\comp{I}}(R)}(\mcalF^I(M),N).
    \end{equation*}
    Hence \(\mcalF^I\dashv\mcalG^I\).
    Its unit is the composite
    \begin{equation} \label{CompletedUnitFormula}
        M\simeq\dgrcomp{I}{M}\xrightarrow{\dgrcomp{I}{\eta_M^0}}\dgrcomp{I}{\mcalG^0\mcalF^0(M)}\xrightarrow{\chi_{\mcalF^0(M)}}\mcalG^I\mcalF^I(M),
    \end{equation}
    because its adjoint is the completion morphism \(\mcalF^0(M)\to\mcalF^I(M)\).
    Apply \(\dgrcomp{I}{-}\) pointwise to the retraction datum of \Cref{NonCompletedSplitUnit} and restrict to complete objects.
    Under \eqref{CofreeCompletionCompatibilityEquation}, the resulting retraction datumis from \(\id\) to \(\mcalG^I\mcalF^I\) whose evaluation at \(M\) is the morphism \eqref{CompletedUnitFormula}.
    It is therefore a retraction datum for \(\eta^I\).
\end{proof}

\begin{corollary} \label{ExistenceRetracts}
    The functor
    \begin{equation*}
        \Pro(\mcalF^I)\colon\Pro(\mcalD_{\graded{G}}^{\comp{I}}(R))\to\Pro(\mcalD^{\comp{I}}(R))
    \end{equation*}
    is conservative, where \(\Pro(\mcalC)\) for an \(\infty\)-category \(\mcalC\) is the pro-completion of \(\mcalC\); see \citeKero{063M}.
\end{corollary}

\begin{proof}
    By the universal property of pro-completion and the composition with the functor to the pro-completion, the retraction datum in \Cref{CofreeCompletionCompatibility} extends to a functor
    \begin{equation*}
        \Ret \to \Fun(\mcalD_{\graded{G}}^{\comp{I}}(R),\Pro(\mcalD_{\graded{G}}^{\comp{I}}(R))) \xleftarrow{\simeq} \Fun'(\Pro(\mcalD_{\graded{G}}^{\comp{I}}(R)), \Pro(\mcalD_{\graded{G}}^{\comp{I}}(R)))
    \end{equation*}
    where the last term is the full subcategory of \(\Fun(\Pro(\mcalD_{\graded{G}}^{\comp{I}}(R)), \Pro(\mcalD_{\graded{G}}^{\comp{I}}(R)))\) consisting of functors preserving cofiltered limits.
    This sends \(0\) to \(\id\) and \(1\) to \(\Pro(\mcalG^I)\Pro(\mcalF^I)\) by again using the universality.
    Therefore, the conservativity follows from the same argument of \Cref{GeneralSplitUnitTheorem}(1).
\end{proof}

\begin{definition} \label{ConstComonad}
    Let \(\mcalT^I\defeq\mcalF^I\mcalG^I\) be the comonad associated with the adjunction of \Cref{CofreeCompletionCompatibility}.
    By \Cref{NonCompletedSplitUnit}, its underlying functor is
    \begin{equation} \label{ComonadRepresentation}
        \mcalT^I(N)\simeq\dcomp{I}{N[G]}\eqdef N\abracket{G} \in \mcalD^{\comp{I}}(R).
    \end{equation}
    The comultiplication is \(\mcalF^I\eta^I\mcalG^I\) and the counit is the counit of the adjunction, with their coherent structure supplied by the opposite form of \cite{lurie2017Higher}*{Proposition 4.7.3.3}.
    Define
    \begin{equation*}
        \coMod_G^{\comp{I}}(R)\defeq\coMod_{\mcalT^I}(\mcalD^{\comp{I}}(R))
    \end{equation*}
    to be the \emph{\(\infty\)-category of derived \(I\)-formal comodules over \(R[G]\)}, with the convention
    \begin{equation*}
        \coMod_{\mcalT}(\mcalD)\defeq(\Mod_{\mcalT^{\opposite}}(\mcalD^{\opposite}))^{\opposite}.
    \end{equation*}
    Write
    \begin{equation} \label{ComparisonFunctorGradedModules}
        (\mcalF^I)'\colon\mcalD_{\graded{G}}^{\comp{I}}(R)\to\coMod_G^{\comp{I}}(R)
    \end{equation}
    for the comparison functor, whose underlying functor is \(\mcalF^I\) and whose coaction on \(\mcalF^I(M)\) is \(\mcalF^I(\eta_M^I)\).
    For a nontrivial grading, \eqref{ComonadRepresentation} uses restriction along \(\rho_R\), not the usual tensor comonad; for the trivial grading, it is the completed tensor comonad \(N\mapsto\dcomp{I}{N\otimes_R^LR[G]}\).
\end{definition}

\begin{remark} \label{RemarkComoduleNotation}
    The notation \(\coMod_G^{\comp{I}}(R)\) does not, in general, mean the \(\infty\)-category of comodules over the derived \(I\)-complete coalgebra \(\dcomp{I}{R[G]}\) in \(\mcalD^{\comp{I}}(R)\).
    Its defining comonad is
    \begin{equation*}
        M\mapsto\dcomp{I}{\rho_{R,*}(M\otimes_R^LR[G])},
    \end{equation*}
    rather than \(M\mapsto M\widehat{\otimes}^L_R\dcomp{I}{R[G]}\).
    When the grading on \(R\) is trivial, \(\rho_R\) is the coefficient inclusion \(r\mapsto rt^0\), so the two comonads agree.
\end{remark}

\begin{corollary} \label{LimitCommutative}
    Let \(\bracelr{M_j}_{j\in J}\colon J\to\mcalD_{\graded{G}}^{\comp{I}}(R)\) be a diagram with a simplicial set \(J\).
    Assume that the limits \(\lim_{j \in J} \dcomp{I}{M_j}\) and \(\lim_{j \in J} (\dcomp{I}{M_j}\abracket{G})\) exist in \(\mcalD^{\comp{I}}(R)\).
    If the canonical morphism
    \begin{equation} \label{AssumptionIsomMorphism}
        (\lim_{j\in J}\dcomp{I}{M_j})\abracket{G}\longrightarrow\lim_{j\in J}\bigl(\dcomp{I}{M_j}\abracket{G}\bigr)
    \end{equation}
    is an isomorphism in \(\mcalD^{\comp{I}}(R)\), then \(\grlim_{j\in J}M_j\) exists in \(\mcalD_{\graded{G}}^{\comp{I}}(R)\) and the canonical morphism
    \begin{equation*}
        \dcomp{I}{\grlim_{j\in J}M_j}\xrightarrow{\cong}\lim_{j\in J}\dcomp{I}{M_j}
    \end{equation*}
    is an isomorphism in \(\mcalD^{\comp{I}}(R)\).
\end{corollary}

\begin{proof}
    By \Cref{PropertiesDerivedGradedModules}\Cref{LimitsColimitsDgrR} and \Cref{PrincipalDerivedGradedwiseCompProp}\Cref{DerivedGrcompLimits}, the category \(\mcalD_{\graded{G}}^{\comp{I}}(R)\) admits small limits, so it is idempotent complete by \cite{lurie2009Higher}*{Corollary 4.4.5.15}.
    The unit of the adjunction of \Cref{CofreeCompletionCompatibility} has a retraction datum.
    Under \eqref{ComonadRepresentation}, \eqref{AssumptionIsomMorphism} is the comparison for \(\mcalT^I=\mcalF^I\mcalG^I\).
    Part~(2) of \Cref{GeneralSplitUnitTheorem} therefore applies.
\end{proof}

\begin{corollary} \label{ComonadicityGradedModules}
    The comparison functor
    \begin{equation*}
        (\mcalF^I)'\colon\mcalD_{\graded{G}}^{\comp{I}}(R)\xrightarrow{\simeq}\coMod_G^{\comp{I}}(R)
    \end{equation*}
    is an equivalence.
    For \(M\in\mcalD_{\graded{G}}^{\comp{I}}(R)\), the canonical equivalence
    \begin{equation} \label{ComparisonCoactionTargetIdentification}
        \mcalT^I(\mcalF^I(M))\simeq\dcomp{I}{\mcalF^0(M)[G]}
    \end{equation}
    identifies its first coaction with the derived \(I\)-completion of the usual grading coaction \(\mcalF^0(M)\to\mcalF^0(M)[G]\).
\end{corollary}

\begin{proof}
    The idempotent completeness verified in the proof of \Cref{LimitCommutative} and the retruction datum of \(\eta^I\) given in \Cref{CofreeCompletionCompatibility} allow us to apply \Cref{SplitUnitComonadicity}.
    This proves the equivalence.
    Applying \(\mcalF^I\) to \(\chi_{\mcalF^0(M)}\), and using \Cref{PrincipalDerivedGradedwiseCompProp}\Cref{DerivedGrcompAdicCompletion}, gives
    \begin{equation*}
        \mcalT^I(\mcalF^I(M))\simeq\dcomp{I}{\mcalF^0\mcalG^0\mcalF^0(M)}\simeq\dcomp{I}{\mcalF^0(M)[G]}.
    \end{equation*}
    By \Cref{NonCompletedSplitUnit}, \(\mcalF^0(\eta_M^0)\) is the grading coaction, whose restriction to \(M_g\) is \(m\mapsto mt^g\).
    Formula \eqref{CompletedUnitFormula} therefore identifies \(\mcalF^I(\eta_M^I)\) with its derived \(I\)-completion under \eqref{ComparisonCoactionTargetIdentification}.
\end{proof}

\end{document}